\documentclass{article}
\usepackage[margin=1in]{geometry}
\usepackage{vmargin}
\usepackage{appendix}
\usepackage[hidelinks]{hyperref}
\usepackage{amssymb,stmaryrd}
\usepackage[shortalphabetic]{amsrefs}
\usepackage{amsmath}
\usepackage{amsrefs}
\usepackage{upgreek}
\usepackage{amsthm}
\usepackage{mathdots}
\usepackage{url}
\usepackage{enumerate}
\usepackage{mathtools}
\usepackage{amsxtra}
\usepackage{xcolor}
 \usepackage[refpage]{nomencl}
\usepackage[shortlabels]{enumitem}
\usepackage{graphicx}
\usepackage{comment}

\makenomenclature

\newtheorem{thm}{Theorem}[section]
\newtheorem{lem}[thm]{Lemma}

\newtheorem{prop}[thm]{Proposition}

\newcommand{\pipsitilde}{\pi^{\thicksim}_{\psi}}
\newcommand{\GL}{\mathrm{GL}}

\newcommand{\Tr}{\mathrm{Tr}}
\newcommand{\Trans}{\mathrm{Trans}}
\newcommand{\R}{\mathbb R}

\newcommand{\C}{\mathbb C}
\newcommand{\etaA}{\eta^{\mathrm{Ar}}}
\newcommand{\etaABV}{\eta^{\mathrm{ABV}}}

\newcommand{\chG}{{}^\vee G}
\newcommand{\ch}[1]{{}^\vee#1}

\newcommand{\X}{ X(\chG^\Gamma)}

\newcommand{\Lift}{\mathrm{Lift}}

\renewcommand{\O}{\ch \mathcal O}
\newcommand{\XO}{ X\left(\O,\chG^\Gamma\right)}
\newcommand{\XiO}{ \Xi\left(\O,\chG^\Gamma\right)}

\title{Arthur packets for pure real forms of symplectic and special
  orthogonal groups} 

\author{N. Arancibia Robert, P. Mezo}

\begin{document}
\maketitle
\begin{abstract}
Arthur packets have been defined for pure real forms of symplectic and
special orthogonal groups following two different approaches.  The
first approach, due to Arthur, Moeglin and Renard uses harmonic
analysis.  The second approach, due to Adams, Barbasch and Vogan uses
microlocal geometry.  We prove that the two approaches produce
essentially equivalent Arthur packets.  This extends previous work of
the authors and J. Adams for the quasisplit real forms.
\end{abstract}
\tableofcontents

\section{Introduction}

The first goal of this note is to review some packets of
representations, conjectured by Arthur, in the setting of
real symplectic and special orthogonal groups.  We
restrict our attention to pure real forms of these groups.  There are
two different approaches to defining these \emph{Arthur packets}.
One of them is due to Arthur himself for
quasisplit forms of these groups (\cite{Arthur}).  His analytic approach was
extended to include pure real forms by Moeglin and
Renard (\cite{MR1}).  The other approach follows the work
of Adams, Barbasch and Vogan in microlocal geometry  (\cite{ABV}). The
principal goal of this 
note is to prove that the two approaches produce essentially
equivalent Arthur packets.  The equivalence of the Arthur packets was given
in the quasisplit case in \cite{AAM}.  The only pure real form of
a symplectic group is quasisplit, and so the truly new
result here is the equivalence of Arthur packets for the pure real
forms of special orthogonal groups. 

We assume that the reader is somewhat familiar with the classification of real
reductive groups, and the basic
formalism of Arthur parameters and Arthur packets. 
Suppose $N \geq 2$ and $G$ is either the complex group $\mathrm{Sp}_N$ or
$\mathrm{SO}_N$.  Fix $\delta_{q}$ to be an antiholomorphic automorphism of
$G$ such that its fixed-point subgroup $G(\mathbb{R}, \delta_{q})$ is
a quasisplit real form of $G$.   Let $\Gamma =
\mathrm{Gal}(\mathbb{C}/\mathbb{R})$.  Then
$\delta_{q}$ determines an action of $\Gamma$ on $G$, and this action
may be identified with the trivial cocycle in $H^{1}(\Gamma,G) =
H^{1}(\mathbb{R},G)$.  As we recall in Section \ref{pureforms},  the
remaining classes 
$\delta \in H^{1}(\Gamma, G)$ also determine real forms $G(\mathbb{R},
\delta)$ of $G$.  We
call the classes in $H^{1}(\Gamma, G)$ the \emph{pure real forms} of
$G$.

Let 
\begin{equation}\label{eq:AparamG}
\psi_G^{}:W_\mathbb{R}\times\mathrm{SL}_2\ \longrightarrow\ {}^{\vee}G^{\Gamma}.
\end{equation}
be an Arthur parameter for $G$, where $^{\vee}G^{\Gamma}$ denotes the
Galois form of the L-group of $G$.  The centralizer
$\mathrm{Cent}(\psi_{G}(W_{\mathbb{R}}),{^\vee}G)$  of the image of
$\psi_{G}$ in ${^\vee}G$ has component group denoted by
$$A_{\psi_{G}} = \mathrm{Cent}(\psi_{G}(W_{\mathbb{R}}),{^\vee}G) \, / \,
(\mathrm{Cent}(\psi_{G}(W_{\mathbb{R}}),{^\vee}G))^{0}.$$
It is a finite 2-group (\cite{Arthur}*{(1.4.8)}).  The Arthur parameter
$\psi_{G}$ also determines an element $s_{\phi_{G}} \in A_{\psi_{G}}$
which is the coset of
\begin{equation} \label{spsi}
\psi_{G}\left(1, \begin{bmatrix} -1 & 0 \\ 0 & -1 \end{bmatrix}
\right) \in \mathrm{Cent}(\psi_{G}(W_{\mathbb{R}}),{^\vee}G).
\end{equation}
In Arthur's approach to defining the packet parameterized by $\psi_{G}$,
there is an additional object that comes into play. It is the group
$\mathrm{Out}_{N}(G)$, which is only non-trivial when
$N$ is even and $G =\mathrm{SO}_{N}$.  In the non-trivial case it is
a group of order two, generated by an outer automorphism of $\mathrm{SO}_{N}$.
(\cite{Arthur}*{\emph{p.}~12}.  See Section \ref{arthurpackets}).

In \cite{Arthur}*{Theorem 2.2.1}, Arthur provides a stable
distribution on the $\mathrm{Out}_{N}(G)$-invariant
test functions on $G(\mathbb{R},\delta_{q})$.  Although it is not
immediately apparent from his notation, his distribution may be
written as
\begin{equation}
  \label{etaar}
\eta^{\mathrm{Ar}}_{\psi_{G}}(\delta_{q}) = \sum_{\tilde{\pi} \in
  \widetilde{\Pi}_{\psi_{G}}^{\mathrm{Ar}}(\delta_{q})} 
 \mathrm{Tr}\left(\tau^{\mathrm{Ar}}_{\psi_{G}}(\tilde{\pi})(s_{\psi_{G}})
 \right) \, \tilde{\pi}
 \end{equation}
(\cite{Arthur}*{(7.4.1)}, \cite{AAM}*{Section 10}).
 Here, $\widetilde{\Pi}_{\psi_{G}}^{\mathrm{Ar}}(\delta_{q})$ is a finite set of
$\mathrm{Out}_{N}(G)$-orbits of irreducible unitary
representations of $G(\mathbb{R}, \delta_{q})$.  We call this
set the \emph{Arthur packet} of $\psi_{G}$ for
$G(\mathbb{R},\delta_{q})$.   In (\ref{etaar}) we identify $\tilde{\pi}$
with its distribution character, and  $\tau_{\psi_{G}}^{\mathrm{Ar}}(\tilde{\pi})$
is a finite-dimensional representation of  $A_{\psi_{G}}$. 
The crux of Arthur's theorem is that the  distributions
$\eta_{\psi_{G}}^{\mathrm{Ar}}(\delta_{q})$  
satisfy both ordinary and twisted endoscopic  identities as
conjectured in \cite{Arthur89}*{Conjecture 6.1 (ii)}.   Arthur
uses harmonic analysis in local and global settings to prove his
theorem.

Continuing in the same vein, Moeglin and Renard extend Arthur's
results to pure real forms (\cite{MR1}).  For every pure real form $\delta \in
H^{1}(\Gamma, G)$, there is a corresponding real form  $G(\mathbb{R},
\delta)$ of $G$.  Moeglin and  
Renard produce stable distributions
$$\eta_{\psi_{G}}^{\mathrm{Ar}}(\delta) = \sum_{\tilde{\pi} \in
  \widetilde{\Pi}_{\psi_{G}}^{\mathrm{Ar}}(\delta)} 
 \mathrm{Tr}\left(\tau^{\mathrm{Ar}}_{\psi_{G}}( \tilde{\pi})(s_{\psi_{G}})
 \right) \, \tilde{\pi}$$
whose terms mirror those of (\ref{etaar}).  In particular, the Arthur
packet $\widetilde{\Pi}_{\psi_G^{}}^{\mathrm{Ar}}(\delta)$ is a finite
set of $\mathrm{Out}_{N}(G)$-orbits of irreducible unitary 
representations of $G(\mathbb{R},\delta)$.
The distributions 
$\eta_{\psi_{G}}^{\mathrm{Ar}}(\delta)$ also satisfy Arthur's conjectured
endoscopic identities, and the method of proof again relies on
harmonic analysis.

Adams, Barbasch and Vogan approach this subject following
different methods, which are based on  microlocal geometry and 
equivariant sheaf theory on a generalized flag variety.  Using these
methods they 
produce stable distributions of the form
$$\eta_{\psi_{G}}^{\mathrm{ABV}}(\delta)=  \sum_{\pi \in
  \Pi_{\psi_{G}}^{\mathrm{ABV}}(\delta)} (-1)^{d S_{\pi} - d S_{\psi_{G}}}\
\mathrm{Tr}\left(\tau^{\mathrm{ABV}}_{\psi_{G}}(\pi) (1)
 \right) \, \pi.$$
The Arthur packet
$\Pi_{\psi_{G}}^{\mathrm{ABV}}(\delta)$ is a finite set of irreducible
representations of $G(\mathbb{R},\delta)$.  \emph{A priori},
these irreducible representations are not known to be unitary
(\emph{cf.} \cite{AMBV}), but they still come attached with finite-dimensional
representations 
$\tau^{\mathrm{ABV}}_{\psi_{G}}(\pi)$ of $A_{\psi_{G}}$.  In addition,
the distributions 
$\eta_{\psi_{G}}^{\mathrm{ABV}}(\delta)$ satisfy
Arthur's conjectured endoscopic identities.
There are  notable advantages to  this approach.  One is that
$\mathrm{Out}_{N}(G)$-orbits do not appear.  Furthermore, the
approach works for \emph{any} connected reductive algebraic group $G$
and \emph{all} of its real forms.

One cannot expect a literal equality between
$\eta_{\psi_{G}}^{\mathrm{ABV}}(\delta)$ and
$\eta_{\psi_{G}}^{\mathrm{Ar}}(\delta)$ since the latter is expressed
in terms of $\mathrm{Out}_{N}(G)$-orbits and the former is not.
Apart from this wrinkle, one would expect the two
distributions to agree.  In the quasisplit setting,
\emph{i.e.}~$\delta = \delta_{q}$, the precise relationship between
$\eta_{\psi_{G}}^{\mathrm{ABV}}(\delta_{q})$ and 
$\eta_{\psi_{G}}^{\mathrm{Ar}}(\delta_{q})$ is given in
\cite{AAM}*{Theorem 9.3}. 
The main result of this note is the extension of this theorem to the
pure real forms, namely
$$\eta_{\psi_G^{}}^{\mathrm{Ar}}(\delta) = 
\mathrm{Out}_N(G)\cdot\left(\eta_{\psi_G^{}}^{\mathrm{ABV}}(\delta)\right),
\quad \widetilde{\Pi}_{\psi_{G}}^{\mathrm{Ar}}(\delta) =
\mathrm{Out}_N(G)\cdot \Pi_{\psi_{G}}^{\mathrm{ABV}}(\delta)$$ 
(Theorem \ref{thm:abv=ar}). In fact, we prove a more refined
identity in which $\eta_{\psi_G^{}}^{\mathrm{Ar}}(\delta)$ and
$\eta_{\psi_G^{}}^{\mathrm{ABV}}(\delta)$ are regarded as
representations of $G(\mathbb{R}, \delta) \times A_{\psi_{G}}$.  This
allows us to extract information about the finite-dimensional
representations $\tau_{\psi_{G}}^{\mathrm{Ar}}(\tilde{\pi})$
and $\tau_{\psi_{G}}^{\mathrm{ABV}}(\pi)$ appearing in each
distribution, and to conclude that
$$\tau^{\mathrm{Ar}}_{\psi_G^{}}(\widetilde{\pi})\, =\,
\tau^{\mathrm{ABV}}_{\psi_G^{}}(\pi)$$ 
where $\widetilde{\pi}$ is the $\mathrm{Out}_N(G)$-orbit of $\pi\in
\Pi_{\psi_{G}}^{\mathrm{ABV}}(\delta)$. Along the way, we work through
several exercises proposed in \cite{AAM}*{Section 10}.

Our work is part of a broader comparison of Arthur packets.  Similar
comparisons have been made for real unitary groups in
\cite{AM-Unitary}, for p-adic general linear groups \cite{CunningRay}, and low
rank p-adic symplectic and special orthogonal groups \cite{cliftonetal}.

\section{Pure real forms and their representations}
\label{pureforms}

In this section $G$ may be any connected complex reductive algebraic group.
We fix  an antiholomorphic involutive automorphism $\delta_{q}$ of $G$
(\cite{ABV}*{(2.1)(b)}) such
that the fixed-point subgroup $G(\mathbb{R}, \delta_{q})$ is a
quasisplit group (\cite{springerart}*{3.2}).  Suppose $\delta$ is a
one-cocyle representing a class in $H^{1}(\Gamma, G)$ as in the
introduction, and let $\sigma$ be the non-identity element in $\Gamma$.
Then $\delta(\sigma)$ is an element in $G$ and the fixed-point
subgroup $G(\mathbb{R},\delta)$ of $\mathrm{Int}(\delta(\sigma))
\delta_{q}$ is a real form of $G$ (\cite{Springer}*{12.3.7}).  We call
$\delta$ a \emph{pure real form}
of $G$.  We often blur the distinction between $\delta$, its class
in $H^{1}(\Gamma, G)$, and (the isomorphism class of) its real form
$G(\mathbb{R}, \delta)$.
The quasisplit real form corresponds to the trivial
cocycle of $H^{1}(\Gamma, G)$ which we shall also denote by
$\delta_{q}$.

We note in passing that it is possible, and
often preferable, to recast the definition of pure real forms so
that $\delta_{q}$ is an \emph{algebraic} automorphism, instead of an
antiholomorphic automorphism (\cite{Adams-Taibi}*{Corollary 4.7}).
In addition, pure real forms are special instances of the \emph{strong
real forms} of \cite{ABV}*{Definition 2.13} and the \emph{rigid inner twists} of
(\cite{kaletha}*{Section 5.2}).

When $G$ is the stabilizer of a bilinear form, as
is the case for symplectic or orthogonal groups, $H^{1}(\Gamma, G)$
is in bijection with the real isomorphism classes of the complex bilinear form
(\cite{bookinv} Proposition 29.1).  If we take $G = \mathrm{Sp}_{N}$
then the quasisplit real form $G(\mathbb{R},\delta_{q}) =
\mathrm{Sp}_{N}(\mathbb{R})$ is split.   
It is well-known that a split real symplectic group
corresponds to a non-degenerate alternating 
bilinear form, and that this bilinear form has only one real isomorphism class 
(\cite{BasicAlgI} Theorem 6.3).
Hence, $H^{1}(\Gamma, \mathrm{Sp}_{N})$ is trivial and the only pure
real form of $\mathrm{Sp}_{N}$ is  the split form
$\mathrm{Sp}_{N}(\mathbb{R})$.  The non-split real forms
$\mathrm{Sp}(p,q)$ are not pure.  

On the other hand, the  orthogonal group $\mathrm{O}_{N}$ is the
  stabilizer of a non-degenerate symmetric 
bilinear form, and its real isomorphism classes are parameterized by $(p,q)$
where $p+q=N$ (\cite{BasicAlgI} Theorem 6.8).  Here, $p$ represents
the $p \times p$ identity matrix, 
$q$ represents the negative of the $q \times q$ identity matrix, and 
$(p,q)$ corresponds to the real form $\mathrm{O}(p,q)$. Fix
$(p',q')$ corresponding to a 
quasisplit orthogonal group $\mathrm{O}(p',q')$ and take $G =
\mathrm{SO}_{N}$.   Then we may identify
$(p',q')$ with $\delta_{q}$ and write $G(\mathbb{R}, \delta_{q}) =
\mathrm{SO}(p',q')$.  The pure real forms relative to $\mathrm{SO}(p',q')$ are
given by those 
$(p,q)$ for which $q$ has the same parity as $q'$
(\cite{bookinv}*{(29.29)}). If $N$ is even
there are exactly two inequivalent choices for $\delta_{q} = (p',q')$:
one corresponding to the 
split form $\mathrm{SO}(N/2,N/2)$, and the other for the  quasisplit
(but not split)  
form $\mathrm{SO}((N/2)+1, (N/2)-1)$.  If $N$ is odd, then the only choice for
$\delta_{q}$ up to equivalence is the one corresponding to the  split form
$\mathrm{SO}((N+1)/2, (N-1)/2)$. 
Going through these possibilities, the conclusion is that 
the pure real forms of $\mathrm{SO}_{N}$ correspond to the groups
$\mathrm{SO}(p,q)$ where $p+q = N$.  For even $N$, the real form
$\mathrm{SO}^{*}(N)$ (\cite{beyond}*{(1.141)}) is not pure. We
therefore avoid the difficulties presented by this real form in
\cite{Arthur}*{Section 9.1}.

Let $\delta \in H^{1}(\mathbb{R},G)$ be a pure real form.  We define a
representation of $G(\mathbb{R}, \delta)$ to be 
an admissible group representation in the sense of
\cite{greenbook}*{Definition 1.1.5}.  
We let $\Pi(G(\mathbb{R}, \delta))$ be the set of infinitesimal
equivalence classes of irreducible representations of
$G(\mathbb{R},\delta)$.  As customary, we will not distinguish between
a representation and its equivalence class.  Every representation in
$\Pi(G(\mathbb{R},\delta))$ has an infinitesimal character, which we
may identify with a ${^\vee}G$-orbit  of a semisimple element in
${^\vee}\mathfrak{g}$ (\cite{ABV}*{Lemma   15.4}).  For a fixed
semisimple ${^\vee}G$-orbit ${^\vee}\mathcal{O} \subset {^\vee}\mathfrak{g}$,
define $\Pi({^\vee}\mathcal{O}, G(\mathbb{R}, \delta))$ to be the subset of
elements in $\Pi(G(\mathbb{R}, \delta))$ with infinitesimal character
${^\vee}\mathcal{O}$.  Let
$$\Pi({^\vee}\mathcal{O}, G/\mathbb{R}) = \coprod_{\delta \in
  H^{1}(\Gamma, G)} \Pi({^\vee}\mathcal{O}, G(\mathbb{R}, \delta))$$
be the disjoint union of the (equivalence classes) of irreducible
representations of all pure real forms 
relative to a fixed quasisplit form $\delta_{q}$.

\section{Arthur Packets}
\label{arthurpackets}

We return to $G$ being a symplectic or special orthogonal group as in the
introduction.  In the following two sections we revisit the two approaches to
defining Arthur packets and their underlying distributions.  The
summaries are meant to provide enough detail to be able to
connect with the work of \cite{AAM} and \cite{AM-Unitary}, where
similar comparisons of the two approaches have been established.  In
the final section we prove the precise identities relating the two
kinds of Arthur packets.

\subsection{The approach of Arthur, Moeglin and Renard} 

We give a review of Arthur's construction of
$\eta_{\psi_{G}}^{\mathrm{Ar}}(\delta_{q})$ in (\ref{etaar}), 
and Moeglin and Renard's extension
$\eta_{\psi_{G}}^{\mathrm{Ar}}(\delta)$ to any pure real form $\delta$.
The starting point of Arthur's approach is to express $G$ as
a \emph{twisted endoscopic group} of $(\mathrm{GL}_{N}, \vartheta)$
(\cite{KS}*{Section 2}, \cite{Arthur}*{Section1.2}).  In this pair, 
$\vartheta$ is the outer automorphism of $\mathrm{GL}_{N}$ defined by
\nomenclature{$\vartheta$}{outer automorphism of $\mathrm{GL}_{N}$}
$$
  \vartheta(g) = \tilde{J} \, (g^{-1})^{\intercal} \, \tilde{J}^{-1}, \quad
  g \in \mathrm{GL}_{N},
$$
where $\tilde{J}$ is the anti-diagonal matrix
\nomenclature{$\tilde{J}$}{anti-diagonal matrix}
\begin{equation*}
\label{tildej}
\tilde{J} = \scriptsize \begin{bmatrix}0 & & & 1\\
  & &-1 & \\
  & \iddots & & \\
  (-1)^{N-1} & & & 0\end{bmatrix}\normalsize .
  \end{equation*}
The semidirect product $\mathrm{GL}_{N} \rtimes \langle \vartheta
\rangle$ is a disconnected algebraic group with non-identity component
$\mathrm{GL}_{N} \rtimes \vartheta$. The group $G$ is attached to the
pair $(\mathrm{GL}_{N}, \vartheta)$ through  an
element $s \vartheta \in
\mathrm{GL}_{N} \rtimes \vartheta$ whose fixed-point set
$({^\vee}\mathrm{GL}_{N})^{s \vartheta}$ contains ${^\vee}G$ as an
open subgroup.  This gives us an inclusion
  ${}^{\vee}G\hookrightarrow {^\vee}\mathrm{GL}_N$,
  which can be extended into an inclusion
$$\mathrm{St}_{G} : {^\vee}G^{\Gamma} \hookrightarrow
{^\vee}\mathrm{GL}_{N}^{\Gamma},$$
that makes $(G,{}^{\vee}G^{\Gamma},s, \mathrm{St}_{G})$ a \emph{twisted
endoscopic datum}  of  $({}^{\vee}\mathrm{GL}_{N}, {}^{\vee}\vartheta)$. 

  The inclusion allows us to extend the Arthur parameter $\psi_G$ of
(\ref{eq:AparamG}) 
into an Arthur parameter
\begin{equation}\label{prepsitilde}
\psi = \mathrm{St}_{G} \circ \psi_{G}
\end{equation}
for $\mathrm{GL}_{N}$. 
 There is a Langlands parameter $\varphi_{\psi}$ associated to $\psi$
 (\cite{Arthur89}*{Section 4}). It is defined by
\begin{equation}
\label{phipsi}
\varphi_{\psi}^{}(w)=\psi \left(w,
\begin{bmatrix}
  |w|^{\frac12}&0\\0&|w|^{-\frac12}
\end{bmatrix} \right), \quad w\in W_\R.
\end{equation}
The local Langlands correspondence attaches to $\varphi_\psi$ 
an irreducible representation $\pi_\psi$ of $\mathrm{GL}_{N}(\R)$.
This representation satisfies $$\pi_\psi\circ\vartheta \cong \pi_\psi,$$ 
and may therefore be extended to a representation 
$\pipsitilde$ of $\GL_N\rtimes\left<\vartheta\right>$.
The extension $\pipsitilde$ is not unique. We choose the extension following
\cite{Arthur}*{\emph{pp.} 62-63} 
by fixing a Whittaker datum. 
Let $\Tr_{\vartheta}(\pipsitilde)$ be the  twisted character of
$\pipsitilde$, \emph{i.e.}~the distribution 
\begin{align*}
\Tr_{\vartheta}(\pipsitilde):\  
 C_{c}^{\infty}(\mathrm{GL}_{N}
  (\mathbb{R})\rtimes \vartheta)\ &\longrightarrow\  \C\\
f \ &\longmapsto\   \mathrm{Tr} \int_{\mathrm{GL}_{N}(\mathbb{R})}
f(x\vartheta) \, \pipsitilde(x) \, \pipsitilde(\vartheta)  \, dx.
\end{align*}

In the same manner, one may define a twisted character
$\mathrm{Tr}(\pi^{\thicksim})$ for 
any extension $\pi^{\thicksim}$ of an irreducible representation $\pi$
of $\mathrm{GL}_{N}(\mathbb{R})$ satisfying $\pi \circ \vartheta
\cong \pi$.  Let $K \Pi(\mathrm{GL}_{N}(\mathbb{R}) \rtimes
\vartheta)$ be the $\mathbb{Z}$-module generated by all such twisted
characters.  For any pure real form $\delta$ of $G$, we define $K
\Pi(G(\mathbb{R}, \delta))$ to be the $\mathbb{Z}$-module generated by
the distribution characters of $\pi \in \Pi(G(\mathbb{R}, \delta))$.
The module $K\Pi(G(\mathbb{R}, \delta))$ is isomorphic to the
Grothendieck group of finite-length representations of $G(\mathbb{R},
\delta)$, and we identify the two modules.

Before introducing Arthur's definition of
$\eta_{\psi_{G}}^{\mathrm{Ar}}(\delta_{q})$, we should return to the group
$\mathrm{Out}_N(G)$ mentioned in the introduction.  The group
$\mathrm{Out}_N(G)$ is defined as the quotient
$$
\mathrm{Out}_N(G)\, =\, \mathrm{Aut}_N(G)/\mathrm{Int}(G)
$$ 
where $\mathrm{Aut}_N(G)$ is the group of automorphisms of the endoscopic
datum $(G,{}^{\vee}G^{\Gamma},s, \mathrm{St}_{G})$ (\cite{KS}*{\emph{p.}~18},
\cite{Arthur}*{\emph{p.}~12}). The group $\mathrm{Out}_N(G)$ may be
identified with the group of outer automorphisms of $G$ (\cite{KS}*{(2.1.8)}).
When $G$ is a symplectic group or an odd rank special orthogonal group
it follows that  $\mathrm{Out}_N(G)$ is trivial.  When $G =
\mathrm{SO}_{N}$ and $N$ is even, then $\mathrm{Out}_N(G) \cong
\mathrm{O}_{N}/ \, \mathrm{SO}_{N}$, a group of order two. In this
case, we choose a canonical representative
$$
w = \tilde{w}(N) \in \mathrm{O}_{N}$$
as in \cite{Arthur}*{\emph{p.}~10}, so that
\begin{equation}\label{eq:representativeON}
\mathrm{Out}_N(G)\, \cong\, \mathrm{O}_{N} / \mathrm{SO}_{N}\, \cong\,
\left<w\right>. 
\end{equation}
The group $G(\mathbb{R}, \delta)$ is preserved by $\mathrm{Out}_{N}(G)$ for any
pure real form $\delta$ (\cite{Arthur}*{Section 9.1}).  Let
$\widetilde{\Pi}(G(\R,\delta))$ be the set of orbits of 
$\mathrm{Out}_N(G)$ in $\Pi(G(\R,\delta))$. 
By definition,
$\widetilde{\Pi}(G(\R,\delta))={\Pi}(G(\R,\delta))$,
unless $G$ is an even rank special orthogonal group.

Suppose $G$ is an even rank special orthogonal group.
Then
$\widetilde{\Pi}(G(\R,\delta))$ contains orbits of cardinality one
or two.   The  restriction of the space of distributions $K \Pi(G(\mathbb{R},
\delta))$ to the $\mathrm{Out}_{N}(G)$-invariant subspace of
$C_{c}^{\infty}(G(\mathbb{R}, \delta))$ is the module of
$\mathrm{Out}_{N}(G)$-coinvariants
\begin{equation} \label{coinv}
K\Pi(G,\mathbb{R},\delta)/
(1-w) \cdot K\Pi(G,\mathbb{R},\delta).
\end{equation}
Here, $w$ acts on a distribution character $\pi \in
\Pi(G(\mathbb{R}, \delta))$ by the transpose action
$$(w \cdot \pi)(f) = \pi\left(\mathrm{Int}(w) \cdot
f \right),  \quad f \in C_{c}^{\infty}(G(\mathbb{R},\delta)),$$
where
$$\mathrm{Int}(w)\cdot f(x) = f(\mathrm{Int}(w)(x)) = f(wxw^{-1}),
\quad x \in G(\mathbb{R},\delta).$$
The actual representation corresponding to the distribution $w \cdot
\pi$ is the usual $w$-conjugate representation.
The module (\ref{coinv}) is isomorphic to $K \widetilde{\Pi}(G(\R,\delta))$, the
$\mathbb{Z}$-module generated by the $\mathrm{Out}_{N}(G)$-orbits
$\widetilde{\Pi}(G(\R,\delta))$.  We identify these two
modules. Actually we are more interested in the complex vector space
$$K_{\mathbb{C}} \widetilde{\Pi}(G(\R,\delta)) = \mathbb{C}
\otimes_{\mathbb{Z}} K \widetilde{\Pi}(G(\R,\delta)).$$
Generally, we shall abbreviate the $\mathbb{C}$-tensor product of any
of our $\mathbb{Z}$-modules with $K_{\mathbb{C}}$.  There are
isomorphisms between the vector spaces
\begin{align*}
  K_{\mathbb{C}}
  \Pi(G(\mathbb{R},\delta))_{|C_{c}^{\infty}(G(\mathbb{R},\delta))^{w}}
  & \cong K_{\mathbb{C}}\widetilde{\Pi}(G(\mathbb{R},\delta)) \\
  & \cong K_{\mathbb{C}} \Pi(G(\mathbb{R},\delta)/ (1-w) \cdot
    K_{\mathbb{C}} \Pi(G(\mathbb{R}, \delta))\\
    & \cong (1/2) (1+w) \cdot K_{\mathbb{C}} \Pi (G(\mathbb{R},\delta)),
  \end{align*}
and  we identify all of them with
$K_{\mathbb{C}}\widetilde{\Pi}(G(\mathbb{R},\delta))$.  Given a
distribution $\pi \in  \Pi (G(\mathbb{R},\delta))$ we
define
\begin{equation} \label{outaction}
  \mathrm{Out}_{N}(G)\cdot \pi \in K_{\mathbb{C}}
  \widetilde{\Pi}(G(\mathbb{R},\delta)) 
  \end{equation}
to mean any of the following equivalent operations
\begin{itemize}
\item Restriction of $\pi$ to the $\mathrm{Out}_{N}(G)$-invariant subspace
  $C_{c}^{\infty}(G(\mathbb{R},\delta))^{w}$

\item The $\mathrm{Out}_{N}(G)$-orbit of $\pi$
\item $(1/2) (1+w) \cdot \pi$.
\end{itemize}
We extend (\ref{outaction}) to $K_{\mathbb{C}} \Pi
(G(\mathbb{R},\delta))$ linearly.  The final interpretation of
(\ref{outaction})  was used
in \cite{AAM}.  We prefer to use the first two here, in line with
\cite{Arthur}.   
When $G$ is not an even rank special orthogonal group then
$\mathrm{Out}_{N}(G)$ is trivial and we define
(\ref{outaction}) to be $\pi$.

Arthur defines the virtual character $\etaA_{\psi_G}(\delta_q)$ 
as the solution of the twisted endoscopic transfer identity
\begin{equation}
  \label{spectrans}
  \Tr_{\vartheta}(\pipsitilde)=\Trans_{G(\R,\delta_q)}^{\GL_N(\mathbb{R}) \rtimes \vartheta} 
  ( \etaA_{\psi_{G}}(\delta_q)),
\end{equation}
with 
\begin{align}\label{eq:twistedtransfer}
\Trans_{G(\R,\delta_q)}^{\GL_N \rtimes \vartheta}:
K_{\mathbb{C}} \Pi(G(\R,\delta_q))^{st}\longrightarrow
K_{\mathbb{C}} \Pi(\mathrm{GL}_N(\R)\rtimes \vartheta),
\end{align}
being the endoscopic transfer map studied in \cite{Mezo}.
The transfer map
$\Trans_{G(\R,\delta_q)}^{\GL_N \rtimes \vartheta}$ is defined on the
space of \emph{stable}  
distributions $K\Pi(G(\R,\delta_q))^{st}$ of $G(\mathbb{R},\delta_q)$.
It is dual to Shelstad's \emph{geometric transfer map}
$C_{c}^{\infty}(\mathrm{GL}_{N}(\mathbb{R} \rtimes \vartheta))
\rightarrow C_{c}^{\infty}(G(\mathbb{R},\delta))$  
(\cite{Shel12}).  Since the image of geometric transfer lies in the
$\mathrm{Out}_{N}(G)$-invariant  
subspace of $C_{c}^{\infty}(G(\mathbb{R}, \delta))$
(\cite{Arthur}*{Corollary 2.1.2}), the endoscopic transfer 
map (\ref{eq:twistedtransfer}) passes to the $\mathrm{Out}_{N}(G)$-coinvariants
of $K_{\mathbb{C}}\Pi(G(\R, \delta_q))^{st}$.  As earlier, we
may regard this space as the image of $\frac{1}{2}(1+w)$ on 
$K_{\mathbb{C}} \Pi(G(\R,\delta_q))^{st}$, or as a subspace of 
$K_{\mathbb{C}} \tilde{\Pi}(G(\mathbb{R}, \delta_{q}))$.  By
\cite{Arthur}*{\emph{pp.}~12, 31}  
and \cite{AMR}*{Proposition 9.1},  
the map $\Trans_{G(\R,\delta_q)}^{\GL_N \rtimes \vartheta}$ 
is injective on this space.  
Consequently, equation (\ref{spectrans}) characterizes
$\etaA_{\psi_G}(\delta_q)$ uniquely, 
and the Arthur packet
$\widetilde{\Pi}_{\psi_G}^{\mathrm{Ar}}(\delta_q)$ is 
defined as the 
$\mathrm{Out}_N(G)$-orbits of irreducible representations occurring in 
$\etaA_{\psi_G}(\delta_q)$.

As things stand, the  distribution $\etaA_{\psi_G}(\delta_q)$ is a
complex linear combination of irreducible unitary representations.  It
turns out that it is in fact a character value of a unitary representation of
$$G(\mathbb{R}, \delta_{q}) \times A_{\psi_{G}}.$$
To see this, we consider the family of endoscopic
groups in the image of $\psi_{G}^{}$ (\cite{Arthur}*{\emph{pp.}~35-37}).
Each endoscopic group in this family corresponds to an element 
$\bar{s} \in A_{\psi_{G}}$  as follows. For each $\bar{s}\in
A_{\psi_G^{}}$ choose a semisimple 
representative $s \in{^\vee}G$, and let $H(\mathbb{R})$ be a quasisplit group
whose dual group 
${^\vee}H$ is the identity component of the centralizer in ${^\vee}G$ of
$s$. Then  
there is a natural embedding (\cite{Arthur}*{(1.4.10)}) 
\begin{align}\label{eq:natural-embeding}
\epsilon:{^{\vee}}H^{\Gamma}\ \hookrightarrow\  {^{\vee}}G^{\Gamma},
\end{align}
and an Arthur parameter $\psi_{H}$ for $H$ such that
$$\psi_{G}^{} = \epsilon \circ \psi_{H}^{}.$$
The group $H$ is a product of symplectic 
and special orthogonal groups. For a more precise description
see  (\ref{eq:H1xH2}) below.
Using Equation (\ref{spectrans}), Arthur attaches 
to each factor in this product a distribution and defines the virtual character 
$\eta_{\psi_{H}}^{\mathrm{Ar}}(H(\R))$ of $H$
to be their product (Equation (\ref{eq:etaAr=tensor-etaAra})). 
For each $\bar{s}\in A_{\psi_G}$,  the virtual character
$\eta_{\psi_G}^{\mathrm{Ar}}(\delta_q)( \bar{s})$ of $G(\R,\delta_q)$
is defined by
\begin{equation}\label{eq:spectraltransferidentity}
\eta_{\psi_G}^{\mathrm{Ar}}(\delta_q)(\overline{s})=
\mathrm{Trans}_{H(\R)}^{G(\R,\delta_q)}\left(\eta_{\psi_{H}}^{\mathrm{Ar}}(H(\R)\right)
\end{equation} 
(\cite{Arthur}*{(2.2.6)}),  where 
\begin{align}\label{eq:ordinarytransfer}
\mathrm{Trans}_{H(\R)}^{G(\R,\delta_q)}:\, K_{\mathbb{C}}
\Pi(H(\R))^{st}\, \longrightarrow\, K_{\mathbb{C}} \Pi(G(\R,\delta)),
\end{align}
is the ordinary endoscopic transfer map on
the space of stable virtual characters of $H(\R)$
(\cite{Langlands-Shelstad}, \cite{Shel82} and \cite{Shel08}).  
This defines a map
$$
\overline{s}\ \longmapsto\ \eta_{\psi_G}^{\mathrm{Ar}}(\delta_{q})(\overline{s}),
\quad \bar{s} \in A_{\psi_{G}}
$$
whose values are complex linear combinations of irreducible unitary
representations of $G(\R,\delta_q)$. One can see that
$\eta_{\psi_G}^{\mathrm{Ar}}(\delta_q)(\cdot)$ is the character
of a finite-length unitary representation of
$G(\R,\delta_q)\times A_{\psi_G^{}}$ in the following manner.  In
\cite{Arthur}*{(7.1.2)}, the virtual character
$\eta_{\psi_{G}}^{\mathrm{Ar}}(\delta_q)(\bar{s})$ is written as 
\begin{equation}\label{7.1.2}
\eta_{\psi_{G}}^{\mathrm{Ar}}(\delta_q)(\bar{s})\, =\, \sum_{\sigma
  \in \tilde{\Sigma}_{\psi_{G}}} < s_{\psi_{G}}\bar{s}, \sigma >
\sigma 
\end{equation}
where $s_{\psi_{G}}$ is given in (\ref{spsi}).
Here, $\tilde{\Sigma}_{\psi_{G}}$ is a
finite set of non-negative integral linear combinations
$$\sigma = \sum_{\widetilde{\pi}} m(\sigma, \widetilde{\pi}) \, \widetilde{\pi}$$
of $\mathrm{Out}_{N}(G)$-orbits of irreducible unitary representations
of $ G(\mathbb{R}, 
\delta_q)$.  As explained in
\cite{Arthur}*{\emph{pp.}~385-386}, there is an injective map from 
$\tilde{\Sigma}_{\psi_{G}}$ into the set of
characters of $A_{\psi_{G}}$ which are trivial on the centre of
${^\vee}G$.  The injection is denoted by
$$\sigma \, \longmapsto\,  < \cdot\, \sigma>.$$  
Following \cite{Arthur}*{Proposition 7.4.3 and (7.4.1)}, we may rewrite
(\ref{7.1.2})  as 
\begin{equation}
\label{7.4.1}
\eta_{\psi_{G}}^{\mathrm{Ar}}(\delta_{q})(\bar{s}) = \sum_{\widetilde{\pi} \in
  \widetilde{\Pi}_{\psi_{G}}^{\mathrm{Ar}}(\delta_{q})} \left( \sum_{\sigma 
  \in \tilde{\Sigma}_{\psi_{G}}}  m(\sigma, \widetilde{\pi}) \,
<s_{\psi_{G}} \bar{s},  
\sigma> \right) \widetilde{\pi}.
\end{equation}
By defining a finite-dimensional representation
\begin{equation}
\label{artprobb}
\tau_{\psi_{G}}^{\mathrm{Ar}} ( \widetilde{\pi}) = \bigoplus_{\sigma \in
  \tilde{\Sigma}_{\psi_{G}}} m(\sigma, \widetilde{\pi}) \, <\cdot, \sigma>,
  \end{equation}
Equation \eqref{7.4.1} becomes
\begin{equation}\label{eq:spectraltransferidentity2}
\eta_{\psi_G^{}}^{\mathrm{Ar}}(\delta_q)(\bar{s})=
\sum_{\widetilde{\pi} \in \widetilde{\Pi}_{\psi_{G}^{}}^{\mathrm{Ar}}(G(\R,\delta_q))}
\mathrm{Tr}\left(\tau_{\psi_{G}^{}}^{\mathrm{Ar}}( \widetilde{\pi} )(s_{\psi_{G}}
\bar{s})\right)\, \widetilde{\pi}. 
\end{equation}
In particular, for $\bar{s}=1$ we see that
$$ 
\eta_{\psi_G}^{\mathrm{Ar}}(\delta_q)\, =\, 
\eta_{\psi_G}^{\mathrm{Ar}}(\delta_q)(1)\, =\,  
\sum_{ \widetilde{\pi} \in \widetilde{\Pi}_{\psi_{G}^{}}^{\mathrm{Ar}}(G(\R,\delta_q))}
\mathrm{Tr}\left(\tau_{\psi_{G}^{}}^{\mathrm{Ar}}(
\widetilde{\pi})(s_{\psi_{G}})\right)\, \widetilde{\pi}.$$ 
Identity (\ref{eq:spectraltransferidentity2})
is called the \textit{ordinary endoscopic transfer identity},
and is one of the main points in Arthur's local conjectures.

In \cite{MR1} and \cite{MR3}, Moeglin and Renard
prove Arthur's local conjectures for all pure real forms $\delta$
 of $G$.  
Following Arthur's approach, for each $\bar{s}\in A_{\psi_G^{}}$
they define $\eta_{\psi_G}^{\mathrm{Ar}}(\delta)(\overline{s})$ as in
(\ref{eq:spectraltransferidentity}). The only difference is that
Kottwitz's sign $e(\delta)$ appears in their definition
(\cite{MR3}*{\S 2.1})
\begin{equation}\label{eq:vHtovG}
\eta_{\psi_G}^{\mathrm{Ar}}(\delta)(\bar{s})\ =\ e(\delta)\,
\mathrm{Trans}_{H(\R)}^{G(\R,\delta)} \left(
\eta_{\psi_{G}}^{\mathrm{Ar}}(H(\mathbb{R}) \right). 
\end{equation}
In \cite{MR1}*{Theorem 9.3} they prove the ordinary endoscopic
transfer identity (\ref{eq:spectraltransferidentity2}) 
 with $\delta_q$ replaced by  $\delta$.   
Thus,  $\eta_{\psi_{G}}^{\mathrm{Ar}}(\delta_{q})$ and the
definition of  Arthur packets are extended to any pure real form
$\delta$ of $G$. 
We point out that, similar to what happens in the quasisplit case, 
the distribution
$\eta_{\psi_G}^{\mathrm{Ar}}(\delta)(\overline{s})$ is only defined up
to the action of $\mathrm{Out}_{N}(G)$.  We also notice that 
for $\bar{s}=1$ the distribution
$$ 
\eta_{\psi_G}^{\mathrm{Ar}}(\delta)\, :=\,
\eta_{\psi_G}^{\mathrm{Ar}}(\delta)(1)\, =\,
\mathrm{Trans}_{G(\R,\delta_{q})}^{G(\R,\delta)} \left(
\eta_{\psi_{G}}^{\mathrm{Ar}}(\delta_{q}) \right)$$
is stable, since $\mathrm{Trans}_{G(\R,\delta_{q})}^{G(\R,\delta)}$
carries stable virtual characters to stable virtual characters.
The Arthur packet 
$\widetilde{\Pi}_{\psi_G^{}}^{\mathrm{Ar}}(\delta)$
is the set of 
$\mathrm{Out}_N(G)$-orbits of irreducible representations occurring in
$\eta_{\psi_G}^{\mathrm{Ar}}(\delta)$. In \cite{MR1}, Moeglin and Renard
use cohomological and parabolic induction to actually give a
description of the representations in each Arthur packet.

\subsection{The approach of Adams, Barbasch and Vogan}

In this section we give a quick review of Adams, Barbasch and Vogan's 
solution $\eta_{\psi_{G}}^{\mathrm{ABV}}(\delta)$ 
to Arthur's conjectures. 
The results in \cite{ABV}  apply to \emph{any}
complex connected reductive group $G$.
However, we continue by assuming that $G$ is symplectic or special
orthogonal, or any finite product of these groups.
We continue by writing $\delta$ for any pure real form of $G$.

Let $P\left({^\vee}G^{\Gamma} \right)$ be the set of
\emph{quasiadmissible} homomorphisms 
$\varphi:W_{\mathbb{R}} \rightarrow {^\vee}G^{\Gamma}$ of $G$
(\cite{ABV}*{Definition 5.2}), \emph{i.e.} the set of L-homomorphisms
for a quasisplit form of $G$. 
Associated to any $\varphi\in P\left({^\vee}G^{\Gamma} \right)$,
there is an infinitesimal character ${^\vee}\mathcal{O} \subset
{^\vee}\mathfrak{g}$  (\cite{ABV}*{Proposition 5.6}). 
Let
\begin{equation}
  \label{qlhomomorphisms}
  P\left({^\vee} \mathcal{O}, {^\vee}G^{\Gamma} \right),
\end{equation}
be the subset of $P\left({^\vee}G^{\Gamma} \right)$ consisting of
homomorphisms with infinitesimal character ${^\vee}\mathcal{O}$. 
The group ${^\vee}G$ acts on
$P\left({^\vee} \mathcal{O},{^\vee}G^{\Gamma} \right)$ by conjugation.   
It is to  the set of ${^\vee}G$-orbits
\begin{equation}
  \label{lparameters}
  P({^\vee}\mathcal{O}, {^\vee}G^{\Gamma})/ {^\vee}G
\end{equation}
that the Langlands correspondence, in its original form  \cite{Langlands}, assigns L-packets of representations
of $G(\R,\delta)$
$$
\varphi\, \longleftrightarrow\, \Pi_\varphi(\delta)
$$
($\Pi_\varphi(\delta)$ is empty when $\varphi$ is not relevant to $\delta$).
A great innovation of \cite{ABV},  
which is central to their construction of Arthur packets,
is to describe the L-packets of $G(\mathbb{R},\delta)$ in terms of an appropriate geometry on ${}^{\vee}G^{\Gamma}$.
This is done through the introduction of the complex variety
$X({^\vee}\mathcal{O}, {^\vee}G^{\Gamma})$ of \emph{geometric
parameters}, which lies between (\ref{qlhomomorphisms}) and
(\ref{lparameters}) (\cite{ABV}*{Definition 6.9}).  It may be regarded
as a set of equivalence classes in $P({^\vee}\mathcal{O},
{^\vee}G^{\Gamma})$ upon which ${^\vee}G$ still acts by conjugation
with finitely many orbits (\cite{AAM}*{Section 2.2},
\cite{ABV}*{Proposition 6.16}).  The map to equivalence classes
$$  P({^\vee}\mathcal{O}, {^\vee}G^{\Gamma}) \longrightarrow
  X({^\vee}\mathcal{O}, {^\vee}G^{\Gamma}) $$
passes to a bijection at the level of ${^\vee}G$-orbits
(\cite{ABV}*{Proposition 6.17}).  Thus, the 
local Langlands correspondence may supplemented by
\begin{equation}\label{eq:LtoS}
 S_\varphi \longleftrightarrow\varphi\longleftrightarrow \Pi_{\varphi}(\delta)
\end{equation}
where  $S_{\varphi} \subset \XO$ is the   ${^\vee}G$-orbit
corresponding to the   ${^\vee}G$-orbit of $\varphi$.
One motivation for introducing $\XO$ is that the closure relations
between ${}^{\vee}G$-orbits imply relationships between the
representations of corresponding L-packets. 
Moreover, 
in \cite{ABV}*{Theorem 10.4} the local Langlands correspondence is refined
from a bijection between L-packets and L-parameters,
to a bijection between (equivalence classes of) irreducible
representations and what they
refer to as \textit{complete geometric parameters}.
More precisely, the authors supplement each ${}^{\vee}G$-orbit
$S\subset X({^\vee}\mathcal{O},{^\vee}G^{\Gamma})$ with a
${}^{\vee}G$-equivariant local system $\mathcal{V}$ 
of vector spaces on $S$, and define 
the set of (pure) complete geometric
parameters $$\Xi({^\vee}\mathcal{O},{^\vee}G^{\Gamma})$$ for
$X({^\vee}\mathcal{O}, {^\vee}G^{\Gamma})$ 
as the set of pairs $\xi=(S,\mathcal{V})$.
The set of ${^\vee}G$-equivariant local systems on $S$
are in bijection with the representations in  an extended L-packet
$$\Pi_{S}(G/\R) \supset \Pi_{\varphi}(\delta)$$
where $S = S_{\varphi}$ as in (\ref{eq:LtoS}).
More symbolically, there is a bijection
\begin{align}\label{eq:RLLC}
\pi(\xi)\longleftrightarrow \xi=(S,\mathcal{V}),
\end{align}
where $\pi(\xi)$ runs over all irreducible representations with
infinitesimal character ${^\vee}\mathcal{O}$ of pure real forms of $G$
including $G(\mathbb{R},\delta)$. 

A striking feature of (\ref{eq:RLLC}) is that the pair
$\xi=(S, \mathcal{V})$ determines a perverse sheaf
$P(\xi)$ on $X({^\vee}\mathcal{O},{^\vee}G^{\Gamma})$. 
Let $\mathcal{P}(X({^\vee}\mathcal{O},{^\vee}G^{\Gamma}))$ be
the category of ${^\vee}G$-equivariant perverse sheaves of
complex vector spaces on $X({^\vee}\mathcal{O},{^\vee}G^{\Gamma})$
(\cite{Lunts}*{Section 5}). This is an abelian category and,  as
explained in \cite{ABV}*{(7.10)(d)}, its simple objects are
parameterized by the 
set of complete geometric parameters  
$$\xi\longmapsto P(\xi),\quad \xi \in
\Xi({^\vee}\mathcal{O}, {^\vee}G^{\Gamma}).$$ 
Hence, we may extend (\ref{eq:RLLC}) to a one-to-one correspondence
\begin{align}\label{eq:RLLC2}
\pi(\xi)\longleftrightarrow \xi\longleftrightarrow P(\xi).
\end{align}
We write $K\Pi({^\vee}\mathcal{O},G/\R)$ and
$KX({^\vee}\mathcal{O},{^\vee}G^{\Gamma})$
for $\mathbb{Z}$-modules generated by
$\Pi({^\vee}\mathcal{O},G/\R)$ and 
$\mathcal{P}(X({^\vee}\mathcal{O},{^\vee}G^{\Gamma}))$  respectively.
These two groups are Grothendieck groups of finite-length objects and
have respective bases
$$
\{\pi(\xi) : \xi \in \Xi({^\vee}\mathcal{O},{^\vee}G^{\Gamma})\}
\quad\text{and}\quad\{P(\xi) : \xi \in
\Xi({^\vee}\mathcal{O},{^\vee}G^{\Gamma})\}.$$
Correspondence (\ref{eq:RLLC2}) expresses a duality between
irreducible representations and irreducible perverse sheaves, in the
form a perfect pairing 
\begin{equation}
\label{prepair}
\langle \cdot, \cdot \rangle_{G} : K \Pi({^\vee}\mathcal{O},G/\R)
\times K \XO \longrightarrow \mathbb{Z}.
\end{equation}
The pairing is defined in \cite{ABV}*{Definition 15.11} using the
alternative bases of standard representations and constructible
sheaves.  It is a deep result (\cite{ABV}*{Theorem 15.12}), that this pairing 
satisfies 
$$\langle \pi(\xi), P(\xi') \rangle_{G} = e(\xi) \, (-1)^{dS_\xi} \,
\delta_{\xi, \xi'}, \quad
\xi,\xi' \in \Xi({^\vee}\mathcal{O},{^\vee} G^{\Gamma}),$$
with $dS_\xi$ the dimension of $S_\xi$, and $\delta_{\xi, \xi'}$ the
Kronecker delta. 

Using pairing \eqref{prepair}, we may regard virtual characters as
$\mathbb{Z}$-valued linear functionals on $K\XO$.
The theory of microlocal geometry provides a family of linear functionals
$$
\chi^{\mathrm{mic}}_{S} : K \XO \longrightarrow \mathbb{Z}
$$
parameterized by the ${^\vee}G$-orbits $S \subset
\XO$.  The linear functional $\chi_{S}^{\mathrm{mic}}$ is called the
\emph{microlocal multiplicity} along $S$. 
The microlocal multiplicities appear in the
theory of \emph{characteristic cycles} (\cite{ABV}*{Chapter 19},
\cite{Boreletal}), and are associated with ${^\vee}G$-equivariant
local systems on a conormal bundle of
$X({^\vee}\mathcal{O},{^\vee}G^{\Gamma})$ (\cite{ABV}*{Section 24},
\cite{GM}*{Appendix 6.A.}). 
The virtual characters associated by the pairing to the linear functionals
$\chi^{\mathrm{mic}}_S$
are stable (\cite{ABV}*{Corollary 1.26, Theorems 1.29 and 1.31}).

The microlocal multiplicity is the last ingredient needed for the
definition of the Arthur packets in \cite{ABV}.  
Suppose $\psi_G$ is an Arthur parameter of $G$, and
let $\varphi_{\psi_G^{}}$ be the L-parameter defined by $\psi_G^{}$
through (\ref{phipsi}). Let ${^\vee}\mathcal{O}$ be the
infinitesimal character of $\varphi_{\psi_{G}}$.  
Then by bijection  (\ref{eq:LtoS}), the parameter
$\varphi_{\psi_G^{}}$ corresponds to a ${}^{\vee}G$-orbit
\begin{equation} \label{SpsiG}
  \varphi_{\psi_G^{}} \longleftrightarrow S_{\psi_{G}}
\end{equation}
in  $\XO$.
The authors   define $\eta^{\mathrm{mic}}_{\psi_G}$ to be the virtual 
character associated to $\chi_{S_{\psi_{G}}}^{\mathrm{mic}}$ using
pairing (\ref{prepair}). That is, 
$\eta^{\mathrm{mic}}_{\psi_G}$ is  the unique virtual character
satisfying
$$\langle \eta^{\mathrm{mic}}_{\psi_{G}}, \mu \rangle_{G} =
\chi_{S_{\psi_{G}}}^{\mathrm{mic}}(\mu), \quad \mu \in K\XO.$$
The stable virtual character $\eta^{\mathrm{mic}}_{\psi_{G}}$ can be expressed
as a linear combination of irreducible
representations of pure real forms of $G$ which include $\delta$. 
Let 
\begin{equation}\label{mmm2}
\eta^{\mathrm{ABV}}_{\psi_{G}}(\delta)
\end{equation}
be the summand of $\eta^{\mathrm{mic}}_{\psi_{G}}$ coming from the
representations in $\Pi({}^{\vee}\mathcal{O},G(\mathbb{R},\delta))$.
The Arthur packet $\Pi_{\psi_G}^{\mathrm{ABV}}(\delta)$ is then defined 
 as  the set of irreducible representations
occurring in
$\etaABV_{\psi_{G}}(\delta)$. More explicitly, 
\begin{equation}
\label{abvdef}
\Pi^{\mathrm{ABV}}_{\psi_{G}}(\delta) \ =\  \{ \pi(\xi) : \xi \in
\Xi({^\vee}\mathcal{O}, {^\vee}G^{\Gamma}),  
\chi^{\mathrm{mic}}_{S_{\psi_{G}}}(P(\xi)) \neq 0, \pi(\xi)\in
\Pi(G(\mathbb{R},\delta)) 
\}.
\end{equation}

It is shown in \cite{ABV}*{Theorem 26.25} that
$\eta_{\psi_{G}}^{\mathrm{ABV}}(\delta)$ satisfies
the ordinary spectral transfer identities
(\ref{eq:spectraltransferidentity2}). In proving
(\ref{eq:spectraltransferidentity2}), 
the authors introduce sheaf-theoretic versions 
of the identities and assert that they are equivalent to their
analytic counterparts.  They begin by giving a sheaf-theoretic version  
of the endoscopic transfer map (\ref{eq:ordinarytransfer}), denoted by
Lift.    It is defined as follows. 
As in the previous section, for each $\bar{s}\in A_{\psi_G^{}}$ choose
a semisimple representative $s \in{^\vee}G$, and let $H(\mathbb{R})$
be its (quasisplit) endoscopic group.
Then  inclusion (\ref{eq:vHtovG}) induces an embedding of varieties
$
\epsilon:X({}^{\vee}\mathcal{O},{^\vee}H^{\Gamma})\hookrightarrow
X({}^{\vee}\mathcal{O},{^\vee}G^{\Gamma})$ (\cite{ABV}*{Corollary
  6.21}).
The inverse image functor  of $\epsilon$,
$$\epsilon^{\ast}:KX({}^{\vee}\mathcal{O},{^\vee}H^{\Gamma})\longrightarrow
KX({}^{\vee}\mathcal{O},{^\vee}G^{\Gamma}),$$ 
 and the pairing (\ref{prepair}), allow one to define a map 
$${\epsilon_{\ast}: K_{\mathbb{C}} \Pi({}^{\vee}\mathcal{O},H/\R)
   \longrightarrow
   K_{\mathbb{C}} \Pi({}^{\vee}\mathcal{O}, G/\R)}$$
(\cite{ABV}*{(26.17)(e)}).
 The Lift map is defined from $\epsilon_{\ast}$ in two steps. First,
 set $\text{Lift}_0$ 
to be the restriction of $\epsilon_{\ast}$ to the subspace
$K_{\mathbb{C}}\Pi({}^{\vee}\mathcal{O},H(\R))^{st}\subset
K_{\mathbb{C}} \Pi({}^{\vee}\mathcal{O},H/\mathbb{R})$ generated
by the stable virtual characters in  
the Grothendieck group $K\Pi({}^{\vee}\mathcal{O},{H}(\mathbb{R}))$ of
representations of $H(\mathbb{R})$. 
The Lift map
$$\text{Lift}_{H(\R)}^{G(\R,
  \delta)}:K\Pi({}^{\vee}\mathcal{O},H(\R))^{st}\longrightarrow 
K\Pi({}^{\vee}\mathcal{O},G(\R,\delta))$$
is then the projection of 
$\text{Lift}_0$ to the Grothendieck group of
representations of $G(\R,\delta)$ (\cite{ABV}*{Definition 26.18}).

It is argued on \cite{ABV}*{\emph{p.}~289} that for the quasisplit
real form $G(\mathbb{R},\delta_{q})$  the sheaf-theoretic and analytic
endoscopic lifting maps are 
the same,
$$\mathrm{Lift}_{H(\R)}^{G(\R,\delta_q)}\, =\, 
\mathrm{Trans}_{H(\R)}^{G(\R,\delta_q)}.$$
In the next section, we will need this identity for any pure real form
$G(\mathbb{R},\delta)$, and for this we appeal to
\cite{AM-Unitary}*{Section 11}, which deals with unitary groups.
The equality of Lift and Trans for  pure real forms of unitary groups
is a consequence of \cite{AM-Lpackets}*{Theorem 1.1}.  Since
\cite{AM-Lpackets}*{Theorem 1.1} is valid for pure real forms of arbitrary
groups $G$, the proof in  \cite{AM-Unitary}*{Section 11} applies
equally well to symplectic and special orthogonal groups.  We
therefore have the following result.
\begin{lem}\label{lem:trans=lift}
For all pure real forms $\delta$ of $G$
\begin{equation*}
\mathrm{Lift}_{H(\R)}^{G(\R,\delta)}=\mathrm{Trans}_{H(\R)}^{G(\R,\delta)}.
\end{equation*}
\end{lem}

Let us go back to the proof of the ordinary endoscopic transfer
identities in \cite{ABV}.
The next step is to introduce a sheaf-theoretic version of the
representation $\tau_{\psi_{G}}^{\mathrm{Ar}}(\tilde{\pi})$ given in
(\ref{artprobb}).  
This is done using deep theorems in microlocal analysis which 
allow  one to express $\chi^{\mathrm{mic}}_{S_{\psi_{G}}}$
as the rank of a local system on a conormal bundle of $\XO$. 
More precisely, it is proven in \cite{ABV}*{Theorem 24.8} that to each
perverse sheaf $P$ on $\XO$ 
there is attached an ${}^{\vee}G$-equivariant local system 
\begin{equation}\label{eq:Qmic}
 Q^{\mathrm{mic}}(P)
\end{equation}
of complex vector spaces on a subset of the
conormal bundle to the ${}^{\vee}G$-action
$T^{\ast}_{{}^{\vee}G}\left(\XO\right)$ (\cite{ABV}*{Equation
  (19.1)(d)}).  Furthermore, for each  
$^{\vee}G$-orbit $S \subset \XO$ the rank of $Q^{\mathrm{mic}}(P)$ at
any non-degenerate point 
$(y,\nu)$ of $T^{\ast}_{S}\left(\XO\right)$ is equal to the microlocal
multiplicity of $P$ along 
$S$, \emph{i.e.}
$$
\dim \left(Q^{\mathrm{mic}}(P)_{y,\nu}\right)\, =\, \chi^{\mathrm{mic}}_{S}(P).
$$
As explained in \cite{ABV}*{Corollary 24.9},
the restriction of $Q^{\mathrm{mic}}(P)$ to $T^{\ast}_{S}(\XO)$
may be represented by a finite-dimensional representation
$\tau^{\mathrm{mic}}_{S}(P)$ 
of the micro-component group (\cite{ABV}*{Definition 24.7})   
$$
A_S^{\mathrm{mic}}={}^{\vee}G_{y,\nu}/\left({}^{\vee}G_{y,\nu}\right)_0,\quad
(y,\nu)\in T^{\ast}_{S}\left(\XO\right), 
$$
verifying
\begin{equation} \label{dimtau}
\dim
\left(\tau_{\psi_G}^{\mathrm{mic}}(P)\right)\ =\ \chi_{S_{\psi_G}}^{\mathrm{mic}}(P).
\end{equation}
We point out that by \cite{ABV}*{Lemma 24.3}, $A_S^{\mathrm{mic}}$ is
independent of the 
choice of $(y,\nu)$. Furthermore, when $S = S_{\psi_{G}}$ as in (\ref{SpsiG}) 
$$
A_{\psi_G}=A_{S_{\psi_G}}^{\mathrm{mic}},
$$
(\cite{ABV}*{Proposition 22.9, Definition 24.7}).
This permits us to define the representation
\begin{equation} \label{ABVprobb}
\tau_{\psi_{G}}^{\mathrm{ABV}}(\pi(\xi)) = \tau_{S_{\psi_G}}^{\mathrm{mic}}(P(\xi)), \quad
\xi \in \Xi({}^{\vee}\mathcal{O},{}^{\vee}G^{\Gamma}).
\end{equation}
With this notation, Equation  (\ref{dimtau}) and pairing
(\ref{prepair}) combine to produce the following decomposition of the
virtual character $\eta^{\mathrm{ABV}}_{\psi_{G}}(\delta)$ of (\ref{mmm2}) 
\begin{equation*}
\label{equationc}
\eta_{\psi_{G}}^{\mathrm{ABV}}(\delta) = \sum_{\pi(\xi) \in 
\Pi_{\psi_{G}}^{\mathrm{ABV}}(\delta)}
(-1)^{dS_\xi - dS_{\psi_{G}}}
\ \dim\left(\tau^{\mathrm{ABV}}_{\psi_{G}}(\pi(\xi)) \right) \, \pi(\xi).
\end{equation*}
Here,  $dS=\dim (S)$ for each ${}^{\vee}G$-orbit $S$. 
In addition, for each $\bar{s}\in A_{\psi_{G}^{}}$ there is a virtual character
\begin{equation}\label{eq:etaABVdef}
\eta_{\psi_{G}}^{\mathrm{ABV}}(\delta)(\bar{s})\  =\ 
\sum_{\pi(\xi) \in \Pi_{\psi_{G}}^{\mathrm{ABV}}(\delta)}
(-1)^{dS_\xi - dS_{\psi_{G}}}
\ \mathrm{Tr}\left(\tau^{\mathrm{ABV}}_{\psi_{G}^{}}(\pi(\xi))(\bar{s})
\right) \, \pi(\xi),
\end{equation}
which resembles $\eta_{\psi_{G}}^{\mathrm{Ar}}(\delta_{q})(\bar{s})$ in
(\ref{eq:spectraltransferidentity2}). 
The ordinary endoscopic transfer identity (\cite{ABV}*{Theorem 26.25})
takes  the form
\begin{equation}
\label{equationd}
\mathrm{Lift}_{H(\R)}^{G(\R,\delta)}\left(\eta_{\psi_{H}}^{\mathrm{ABV}}(H(\R))\right)
= \eta_{\psi_{G}}^{\mathrm{ABV}}(\delta)(\bar{s}). 
\end{equation}

It is natural to seek a sheaf-theoretic analogue of the twisted
endoscopic transfer identity (\ref{spectrans}) as well.  This is
proven in \cite{AAM}. Indeed, following \cite{ABV},
a sheaf-theoretic version of the theory of twisted endoscopy is
introduced  in \cite{Christie-Mezo}.  This is  used in
\cite{AAM}*{Equation (134)} to give a sheaf-theoretic version of the
twisted transfer map (\ref{eq:twistedtransfer})
$$\text{Lift}_{G(\R,\delta_q)}^{\mathrm{GL}_N(\R)\rtimes \vartheta}:
K_{\mathbb{C}} \Pi(G(\R,\delta_q))^{st} \longrightarrow
K_{\mathbb{C}} \Pi(\mathrm{GL}_N(\R)\rtimes \vartheta).$$
The twisted endoscopic transfer identity takes the form
$$
\mathrm{Lift}_{{G(\R,\delta_q)}}^{\mathrm{GL}_N(\R)\rtimes
  \vartheta}\left(\eta_{\psi_{G}}^{\mathrm{ABV}}(\delta_q)\right) =
\mathrm{Tr}_{\vartheta} \left( 
\eta_{\psi}^{\mathrm{ABV}}(\delta_q)^{{\thicksim}} \right), 
$$
where $\psi$ is as in (\ref{prepsitilde}), 
and $\eta_{\psi}^{\mathrm{ABV}}(\delta_q)^{{\thicksim}}$ is the virtual
character obtained from $\eta_{\psi}^{\mathrm{ABV}}(\delta_q)$
after canonically extending each $\pi\in
\Pi_{\psi}^{\mathrm{ABV}}(\mathrm{GL}_N(\R))$ to a representation of  
$\mathrm{GL}_N(\R)\times\left<\vartheta\right>$.


We end this section by assuming that
$G= \mathrm{SO}_{N}$, $N$ even,  and describing the effect of the action of
$\mathrm{Out}_N(G)$ on the  objects involved in the
definition of the Arthur packets in \cite{ABV}. 
As in the previous section, $\mathrm{Out}_N(G)\cong\left<w\right>$ 
with $w \in \mathrm{O}_{N}$ as in (\ref{eq:representativeON}).
The outer automorphism $\mathrm{Int}(w)$ on $G$
induces a natural bijection
\begin{equation}\label{eq:maponparameters}
\mathrm{Int}({w}):\, X\left(\O,
{^\vee}G^{\Gamma}\right) \, \longrightarrow\,  X\left({w}\cdot \O,
{^\vee}G^{\Gamma}\right),
\end{equation}
which by \cite{ABV}*{Proposition 7.15(a)},
induces a bijection of complete geometric parameters 
$$
w^{\ast}: 
 \Xi\left({w}\cdot \O,
{^\vee}G^{\Gamma}\right)\,  \longrightarrow\, 
\Xi\left(\O,
{^\vee}G^{\Gamma}\right).
$$
Moreover, by \cite{ABV}*{Proposition 7.15(b)} the inverse image
functor attached to  
(\ref{eq:maponparameters})
$$
\mathrm{Int}({w})^{\ast}:\mathcal{P}\left(X\left({w}\cdot \O,
{^\vee}G^{\Gamma}\right)\right)\, \longrightarrow\,  \mathcal{P}\left(X\left(\O,
{^\vee}G^{\Gamma}\right)\right)
$$
is a fully faithful exact functor, satisfying 
\begin{equation}\label{eq:inverseimage-on-irrperverse}
\mathrm{Int}({w})^{\ast}(P(\xi))=P(w^{\ast}\xi).
\end{equation}
Two more maps are induced by (\ref{eq:maponparameters}). The 
differential $\mathrm{Ad}(w)$ of (\ref{eq:maponparameters}) defines a
homeomorphism between the tangent bundles 
$$\mathrm{Ad}({w}) : T  \left( X\left(\O,
{^\vee}G^{\Gamma}\right) \right)  \, \longrightarrow \, T \left( X\left({w}\cdot
\O,{^\vee}G^{\Gamma}\right) \right), $$
and duality gives us a homeomorphism between the 
corresponding conormal bundles
$$\mathrm{Ad}^{*}({w}):T_{{^\vee}{G}}^{\ast} \left( X
\left(w\cdot\O,{^\vee}G^{\Gamma} \right) \right) \, \longrightarrow 
\, T_{{}^{\vee}{G}}^{\ast} \left( X\left(\O,{^\vee}G^{\Gamma}\right) \right)$$
such that for any ${}^{\vee}G$-orbit $S \subset \XO$, we have
$$\mathrm{Ad}^{*}({w}) \left(T_{\mathrm{Int}(w) S}^{\ast} \left(
X\left({w}\cdot\O_{G}, {^\vee}G^{\Gamma}\right)\right) \right)= 
T_S^{\ast} \left( X\left(\O_{G},{^\vee}G^{\Gamma}\right) \right).$$
We can now describe the effect of $\mathrm{Out}(G)$ on ABV-packets. 
\begin{prop}\label{prop:OutGofPi}
Suppose $\psi_G$ is an Arthur parameter
of $G = \mathrm{SO}_{N}$, and  $N$ is even. Let $w$ be as in
(\ref{eq:representativeON}). 
Then
\begin{enumerate}[$($a$)$]
\item
$\tau^{\mathrm{ABV}}_{\mathrm{Int}({w})\,\circ\,\psi_{G}^{}}(\pi(w^{\ast}\xi))\, =\,
\tau^{\mathrm{ABV}}_{\psi_{G}^{}}(\pi(\xi))\circ \mathrm{Int}({w}),
\quad  \xi \in\XiO$ 

\item  \begin{align*}
\Pi_{\mathrm{Int}({w})\,\circ\,\psi_{G}^{}}^{\mathrm{ABV}}(\delta)\, &=\, 
\left\{\pi(w^{\ast}\xi)\ :\ 
\xi \in\XiO,\ 
\pi(\xi)\in \Pi_{\psi_{G}^{}}^{\mathrm{ABV}}(\delta)\right\}\nonumber\\
\, &=\, 
\left\{\pi(\xi)\circ \mathrm{Int}({w})\ :\ \xi \in\XiO,\ \pi(\xi)\in
\Pi_{\psi_{G}^{}}^{\mathrm{ABV}}(\delta)\right\}
\end{align*}

\item  $w \cdot \eta_{\psi_{G}}^{\mathrm{ABV}}(\delta)(\overline{s})
= \eta_{\mathrm{Int}(w) \circ  \psi_{G}}^{\mathrm{ABV}} (\delta)
(\mathrm{Int}(w)(\overline{s})).$
\end{enumerate}
\end{prop}
\begin{proof}
  Let $\xi\in \Xi({^\vee}\mathcal{O},{^\vee}G^{\Gamma})$ and let
  $Q^{\mathrm{mic}}(P(\xi))$ be the ${}^{\vee}G$-equivariant local system on 
a subset of $T^{\ast}_{{}^{\vee}G}\left(\XO\right)$ as in (\ref{eq:Qmic}).  
The stalks
of $ Q^{\mathrm{mic}}(P(\xi))$ at a point $(y,\nu)\in
T^{\ast}_{S}\left(\XO\right)$ are given in \cite{ABV}*{(24.10)(b)} by
the relative hypercohomology  
$$
Q^\mathrm{mic}(P(\xi))_{y,\nu}=H^{-\dim S}(J,K;P(\xi)).
$$
Here, $(J,K)$ is a pair of compact subspaces
of $\XO$  with $K\subset J$ (\cite{ABV}*{(24.10) (a)}). Let 
$h:\, J-K\, \hookrightarrow\, J$  be the 
inclusion and let $P(\xi)_{|J}$ be the restriction of $P(\xi)$ to $J$. 
By definition
$$
H^{\ast}\left(J,K;P(\xi)\right)=
\mathbb{H}^{\ast}\left(J;Rh_{!}\, h{^!}P(\xi)_{|J} \right).
$$
We wish to transfer these objects using the
homeomorphisms induced by $w= w^{-1}$.  Let us write
$w\cdot h \cdot w\, =\, \mathrm{Int}(w)\circ h \circ \mathrm{Int}(w)$
for short.  Then clearly $w \cdot h \cdot w$ is the inclusion of the
$w$-conjugate of $J-K$ into the  $w$-conjugate of $J$.
Together with (\ref{eq:inverseimage-on-irrperverse}) we deduce
\begin{align*}
&  Q^\mathrm{mic}(P(w^{\ast}\xi))_{\mathrm{Int}(w)(y),\mathrm{Ad}^{*}(w)(\nu)}\\
  &=Q^\mathrm{mic} \left( \mathrm{Int}(w)^{\ast}P(\xi)
  \right)_{\mathrm{Int}(w)(y),\mathrm{Ad}^{*}(w)(\nu)}\\ 
&= H^{-\dim ({\mathrm{Int}(w)S})} \left( \mathrm{Int}(w)(J), \mathrm{Int}(w)(K);
\mathrm{Int}(w)^{\ast}P(\xi)_{|J}  \right)\\
&= \mathbb{H}^{-\dim ({\mathrm{Int}(w)S})} \left(
\mathrm{Int}(w)(J); R(w\cdot h\cdot w)_{!}(w\cdot h\cdot w){^!}
\mathrm{Int}(w)^{\ast}P(\xi)_{|J} \right).
\end{align*}
In order to simplify this expression we make some observations about
the functors which appear in it.  Since $\mathrm{Int}(w)$ is a
homeomorphism satisfying 
$\left(\mathrm{Int}(w)\right)^2=\mathrm{Id}$, it follows that
$$\mathrm{Int}(w)_{!} = \mathrm{Int}(w)_{*},\quad \mathrm{Int}(w)^{*} =
\mathrm{Int}(w)^{!},\quad \ R\mathrm{Int}(w)_{*}\,
\mathrm{Int}(w)^{*} = \mathrm{Id}$$
as long as one keeps track of the domains and codomains of the
functors.  Additionally,
since both $h$ and $\mathrm{Int}(w)$ are open embeddings
$$
R(w\cdot h \cdot w)_{!} =
R\mathrm{Int}(w)_!\circ Rh_! \circ R\mathrm{Int}(w)_!  
 =\mathrm{Int}(w)_\ast \circ Rh_! \circ \mathrm{Int}(w)_\ast$$
and
$$(w\cdot h \cdot w)^! =
\mathrm{Int}(w)^!\circ h^! \circ \mathrm{Int}(w)^!  
\, =\,
\mathrm{Int}(w)^\ast\circ h^! \circ \mathrm{Int}(w)^\ast.$$
(\cite{Achar}*{Proposition 1.3.7}).
Hence,
\begin{align*}
  & \mathbb{H}^{-\dim  ({\mathrm{Int}(w)S})}
  \left(\mathrm{Int}(w)(J); R(w\cdot h\cdot w)_{!}(w\cdot h\cdot w){^!}
  \mathrm{Int}(w)^{\ast}P(\xi)_{|J} \right)\\ 
&=
\mathbb{H}^{-\dim ({\mathrm{Int}(w)S})} \left(
\mathrm{Int}(w)(J); R\mathrm{Int}(w)_\ast Rh_!h^!P(\xi)_{|J} \right)\\
&\cong
\mathbb{H}^{-\dim ({\mathrm{Int}(w)S})}\left(J; Rh_!   h^!P(\xi)_{|J} \right)\\
&= H^{-\dim (S)}(J,K; P(\xi))\\
&=Q^{\mathrm{mic}}(P(\xi))_{y,\nu}.
\end{align*}
This implies that
\begin{equation}\label{transQmic}
(\mathrm{Ad}^{*}(w))^{\ast} Q^\mathrm{mic}(P(\xi))\, =\, 
Q^\mathrm{mic}(P(w^{\ast}\xi)),
\end{equation}
where $(\mathrm{Ad}^{*}(w))^{\ast}$ is the inverse image functor of
$\mathrm{Ad}^{*}(w)$.  
The representation $\tau_{\psi_{G}}^{\mathrm{ABV}}(\pi(\xi))$ of
$A_{\psi_{G}}$ is determined by the local system $Q^{\mathrm{mic}}(P(\xi))$ on the
left.  An application of 
\cite{ABV}*{Proposition 7.18} to (\ref{transQmic}) tells us that  
the representation of $A_{\mathrm{Int}(w)\circ\psi_G}$ determined by
the local system $Q^\mathrm{mic}(P(w^{\ast}\xi))$ is
$\tau^{\mathrm{ABV}}_{\psi_{G}^{}}(\pi(\xi))\circ \mathrm{Int}(w)$.   This
proves (a).

For part (b) we observe that
\begin{align*}
  \chi^{\mathrm{mic}}_{S_{\mathrm{Int}(w)\circ\psi_G}}(P(w^{\ast}\xi))\ &=\
  \dim \left(
Q^{\mathrm{mic}}(P(w^{\ast}\xi))_{\mathrm{Int}(w)(y),\mathrm{Ad}^{*}(w)(\nu)}
\right)\\
&= \   \dim \left(Q^{\mathrm{mic}}(P(\xi))_{y,\nu}\right) \\
&=\ \chi^{\mathrm{mic}}_{S_{\psi_G}}(P(\xi)).
\end{align*}
Assertion (b) now follows from  definition (\ref{abvdef}).

Finally, by parts (a) and (b), we may write 
\begin{align*}
w \cdot \eta_{\psi_{G}}^{\mathrm{ABV}}(\delta)(\overline{s})\ 
&= 
\sum_{\pi(\xi) \in \Pi_{\psi_{G}}^{\mathrm{ABV}}(\delta)}
(-1)^{dS_\xi - dS_{\psi_{G}}}
\ \mathrm{Tr}\left(\tau^{\mathrm{ABV}}_{\psi_{G}^{}}(\pi(\xi))\circ
\mathrm{Int}(w)\left(\mathrm{Int}(w)(\bar{s})\right) \right) \,
w \cdot \pi(\xi)\\  
&=\,
\sum_{\pi(\xi) \in \Pi_{\psi_{G}}^{\mathrm{ABV}}(\delta)}
(-1)^{dS_{w^\ast\xi} - dS_{\mathrm{Int}(w)\circ \psi_G}}
\ \mathrm{Tr} \left(
\tau^{\mathrm{ABV}}_{\mathrm{Int}(w)\,\circ\,\psi_{G}^{}}
(\pi(w^{\ast}\xi))\left(\mathrm{Int}(w)(\overline{s}) \right) 
\right) \, \pi(w^{\ast}\xi)\\  
 &=\, \eta_{\mathrm{Int}(w)\circ
  \psi_G}^{\mathrm{ABV}}(\delta)(\mathrm{Int}(w)(\overline{s})). 
\end{align*}
\end{proof}

\subsection{The comparison of the approaches}

We wish to compare
$\eta^{\mathrm{Ar}}_{\psi_G^{}}(\delta)(\bar{s})$, defined in
(\ref{eq:vHtovG}), 
with $\eta^{\mathrm{ABV}}_{\psi_G^{}}(\delta)(\bar{s})$, defined in
(\ref{eq:etaABVdef}).  When $G$ is an even rank special orthogonal
group, the  former distribution is only defined on 
$\mathrm{Out}_{N}(G)$-orbits (\ref{eq:representativeON}), and so the
most we can hope for is an 
identity of the form
\begin{equation} \label{eq1}
\eta^{\mathrm{Ar}}_{\psi_G^{}}(\delta)(\bar{s}) = \mathrm{Out}_{N}(G)
\cdot \eta^{\mathrm{ABV}}_{\psi_G^{}}(\delta)(\bar{s}).
\end{equation}
The proof of this identity is the goal of this section.  Once this is
established, the identity 
$$\widetilde{\Pi}_{\psi_G^{}}^{\mathrm{Ar}}(\delta)\, =\,
\mathrm{Out}_N(G)\cdot \Pi_{\psi_G^{}}^{\mathrm{ABV}}(\delta)$$
is immediate.  It is then also not difficult to prove  
$$\quad \tau^{\mathrm{Ar}}_{\psi_G^{}}(\widetilde{\pi})=
\tau^{\mathrm{ABV}}_{\psi_G^{}}(\pi), \quad \pi \in
\Pi_{\psi_G^{}}^{\mathrm{ABV}}(\delta) $$
where $\widetilde{\pi} = \mathrm{Out}_{N}(G) \cdot
\pi$ (\emph{cf.}~(\ref{artprobb}) and (\ref{ABVprobb})).

The main objective of \cite{AAM} was to prove these identities in the case of 
$\delta = \delta_q$ a quasisplit pure real form of $G$.
This was done by first proving that the two versions of twisted endoscopic
transfer agreed (on $\mathrm{Out}_{N}(G)$-orbits), \emph{i.e.}
$\Lift_{{G(\R,\delta_q)}}^{\GL_N(\mathbb{R}) \rtimes 
  \vartheta}=\Trans_{{G(\R,\delta_q)}}^{\GL_N(\mathbb{R}) \rtimes
  \vartheta}$
(\cite{AAM}*{Corollary 7.10}).  Then it was proved that 
$$\mathrm{Lift}_{{G(\R,\delta_q)}}^{\mathrm{GL}_N(\R)\rtimes \vartheta}
\left(\eta_{\psi_{G}}^{\mathrm{ABV}}(\delta_q)\right) =
\mathrm{Tr}_{\vartheta} \left(
\eta_{\psi}^{\mathrm{ABV}}(\delta_q)^{\thicksim} \right)  =
\mathrm{Tr}_{\vartheta} \left( \pi_{\psi}^{\thicksim} \right) =
\mathrm{Lift}_{{G(\R,\delta_q)}}^{\mathrm{GL}_N(\R)\rtimes \vartheta}
\left(\eta_{\psi_{G}}^{\mathrm{Ar}}(\delta_q)\right)$$
(\cite{AAM}*{Proposition 6.3, Theorem 9.3}). By using the injectivity
of  $\mathrm{Lift}_{{G(\R,\delta_q)}}^{\mathrm{GL}_N(\R)\rtimes
  \vartheta}$ on $\mathrm{Out}_{N}(G)$-orbits one obtains 
(\ref{eq1})  for $\delta= \delta_{q}$ and $\bar{s} =
1$.  The extension to arbitrary $\bar{s} \in A_{\psi_{G}}$ follows from
a comparison of ordinary endoscopy (\cite{AAM}*{Section 10}).
In what follows, we will review the comparison of ordinary endoscopy
and continue using it to achieve identity (\ref{eq1})  for
arbitrary pure real forms $\delta$ of $G$.

Let $\bar{s}\in A_{\psi_G^{}}$ and fix  a semisimple representative
$s\in {}^{\vee}G$ of $\bar{s}$. As in the previous sections, we write
$H(\mathbb{R})$  for the quasisplit endoscopic group  whose 
dual ${^\vee}H$ is 
the identity component of the centralizer in ${^\vee}G$ of $s$. Recall
that using  (\ref{eq:natural-embeding}) there is an Arthur parameter
$\psi_{H}$ for $H$ such that $\psi_{G}^{} = \epsilon \circ \psi_{H}^{}.$

An intermediate step towards to proving (\ref{eq1})
is to prove analogous  identities for the endoscopic group
$H(\mathbb{R})$. More precisely, we wish to prove an identity of the form   
\begin{align}\label{eq:etaAr=etaABVodd}
\eta_{\psi_H}^{\mathrm{Ar}}(H(\R,\delta_q)) \,=\, \left(
\times_{i=1,2}\mathrm{Out}_{N_i}(H_i) \right)\cdot
\eta_{\psi_H}^{\mathrm{ABV}}(H(\mathbb{R})), 
\end{align}  
with $H_i,\ i=1,2$ defined in Equation (\ref{eq:H1xH2}) below. 
In \cite{AAM}*{Section 10} this identity is verified in the case of $G
= \mathrm{SO}_N$, $N$ odd.  The other cases were left as exercises.
We complete these exercises here.  

Suppose $N$ is even and
$G$ equals to  $\mathrm{Sp}_N$ or $\mathrm{SO}_N$. The explicit
description in \cite{Arthur}*{(1.4.8)} 
of the centralizer  in ${^\vee}G$ of the image of $\psi_{G}$ 
makes it clear that every element $\bar{s} \in A_{\psi_{G}}$ has a
diagonal representative $s$ in the centralizer with eigenvalues $\pm
1$.   Hence, as mentioned in  
\cite{Arthur}*{\emph{pp.} 13-14},
the quasisplit endoscopic group $H(\mathbb{R})$ determined by $s$ is a direct
product $$H_1(\mathbb{R}) \times H_2(\mathbb{R}),$$  
in which 
\begin{align}\label{eq:H1xH2}
H_1\, =\, \mathrm{SO}_{N_1},\quad  
H_2\, =\, \left\{
\begin{array}{cl}
\mathrm{Sp}_{N_2}&\text{if }G=\mathrm{Sp}_N,\\
\mathrm{SO}_{N_2}&\text{if }G=\mathrm{SO}_N,
\end{array}
\right.\quad \ N_1,\, N_2\, \text{ even},\ N_1+N_2=N.
\end{align}
The following additional conditions hold:
\begin{itemize}  
\item[$\bullet$]For $G(\R,\delta_q)=\mathrm{Sp}_N(\R,\delta_q)$, the
  quasisplit group $H_1(\mathbb{R})= \mathrm{SO}_{N_1}(\R,\delta_q)$ can be
  either split or non-split.  
\item[$\bullet$] For $G(\R,\delta_q)=\mathrm{SO}_N(\R,\delta_q)$
  split, the quasisplit groups $H_i(\mathbb{R})=
  \mathrm{SO}_{N_i}(\R,\delta_q)$,  $i=1,2,$ are both split or both
  non-split. 
\item[$\bullet$] For $G(\R,\delta_q)=\mathrm{SO}_N(\R,\delta_q)$
  non-split, one of the two quasisplit groups $H_i(\mathbb{R})=
  \mathrm{SO}_{N_i}(\R,\delta_q)$, $i=1,2,$ is split and the other
  is non-split. 
\end{itemize}

Following \cite{Arthur}*{pp. 31, 36}, the Arthur parameter
$\psi_{H}^{}$ decomposes  as a product $\psi_{H_1}^{} \times
\psi_{H_2}^{}$ of Arthur parameters.
By \cite{Arthur}*{Theorem 2.2.1 (a)} the stable virtual character 
$\eta_{\psi_{H}}^{}(H(\R))$ is defined as the tensor product 
\begin{equation}\label{eq:etaAr=tensor-etaAra}
\eta_{\psi_{H}}^{\mathrm{Ar}}(H(\mathbb{R})) =
\eta_{\psi_{H_{1}}}^{\mathrm{Ar}} (H_{1}(\mathbb{R}))
\otimes \eta_{\psi_{H_{2}}}^{\mathrm{Ar}}(H_{2}(\mathbb{R})), 
\end{equation}    
where we recall that $\eta_{\psi_{H_i}}^{\mathrm{Ar}}(H_i(\R))$, 
for $H_{i}$ an even rank special orthogonal group, is defined as an orbit under
$\mathrm{Out}_{N_i}(H_i)\cong\mathrm{O}_{N_i}/\mathrm{SO}_{N_i}$.  This  
 is to say,  that  $\eta_{\psi_{H_i}}^{\mathrm{Ar}}(H_i(\R))$ is
 invariant under the action of the outer automorphisms 
 induced by the orthogonal group $\mathrm{O}_{N_i}$.  
We recall also that
$\mathrm{Out}_{N_2}(H_2)$ is the trivial group for $H_2$ a symplectic
group.

Let us move to the right-hand side of (\ref{eq:etaAr=etaABVodd}). 
An argument similar to the one implemented in the proof of
\cite{AAM}*{Corollary 6.2}  
permits us to obtain a decomposition 
\begin{align}\label{eq:eta=tensor-eta}
\eta_{\psi_{H}}^{\mathrm{ABV}}(H(\R))\, =\,  
\eta_{\psi_{H_1}^{}}^{\mathrm{ABV}}(H_1(\R)) \otimes
\eta_{\psi_{H_2}^{}}^{\mathrm{ABV}}(H_2(\R)). 
\end{align}

Identity (\ref{eq:etaAr=etaABVodd}) follows from  the
previous two  identities and \cite{AAM}*{Theorem 9.3}. Indeed, for
$H_2$ a symplectic group,  \cite{AAM}*{Theorem 9.3 (a)} tells us that 
$$\eta_{\psi_{H_2}^{}}^{\mathrm{Ar}}(H_2(\R)) \,=\,
\eta_{\psi_{H_2}^{}}^{\mathrm{ABV}} (H_2(\R)).$$
For $H_i,\, i=1,2$ an even rank special orthogonal group, 
Proposition \ref{prop:OutGofPi} (c)  with $\overline{s}=1$ tells us that
$$
\mathrm{Out}_{N_i}(H_i)\cdot \eta_{\psi_{H_i}^{}}^{\mathrm{ABV}}(H_i(\R))\ =\ \left\{
\eta_{\psi_{H_i}^{}}^{\mathrm{ABV}}(H_i(\R)),
\eta_{{^{\vee}}\mathrm{Int}(w)\circ \psi_{H_i}^{}}^{\mathrm{ABV}}(H_i(\R))
\right\}
$$
and from \cite{AAM}*{Theorem 9.3 (b)} we conclude
$$
\eta_{\psi_{H_i}^{}}^{\mathrm{Ar}}(H_i(\R)) \,=\,
\mathrm{Out}_{N_i}(H_i)\cdot
\eta_{\psi_{H_i}^{}}^{\mathrm{ABV}}(H_i(\R)). 
$$
Identity (\ref{eq:etaAr=etaABVodd}) is now immediate from the
decompositions of 
$\eta_{\psi_{H}^{}}^{\mathrm{Ar}}(H(\R))$ and 
$\eta_{\psi_{H}^{}}^{\mathrm{ABV}}(H(\R))$
given in equations
(\ref{eq:etaAr=tensor-etaAra})
and (\ref{eq:eta=tensor-eta}) respectively.

The comparison of distributions in  (\ref{eq:etaAr=etaABVodd}) for
endoscopic groups can be lifted to a comparison of distributions for
$G(\mathbb{R},\delta)$ by using ordinary endoscopic transfer.  The
details are given in  the following lemma.
\begin{lem}\label{lem:Out-of-eta} Let $\psi_G$ be an Arthur parameter for $G$. 
For any $\overline{s}\in A_{\psi_G}$, choose $H$ and $\psi_H$ as in
the beginning of this section. Then  
\begin{align}\label{eq:wLift=Liftw}
\mathrm{Lift}_{H(\mathbb{R})}^{G(\mathbb{R},\delta)} \left( \left(
\times_{i=1,2}\mathrm{Out}_{N_i}(H_i)\right)  \cdot 
\eta_{\psi_{H}^{}}^{\mathrm{ABV}}(H(\R))\right) \ =\  
\mathrm{Out}_N(G)\cdot \eta_{\psi_{G}}^{\mathrm{ABV}}(\delta)(\overline{s}).
\end{align}
\end{lem}
\begin{proof} The proof is done case-by-case, depending on the form of $G$.
\begin{enumerate} 
\item[a.] For $G=\mathrm{SO}_N$, $N$ odd, the groups $H_{1}$, $H_{2}$
  are odd rank special orthogonal groups
  (\cite{Arthur}*{\emph{p.}~13}).  Therefore, 
  $\times_{i=1,2}\mathrm{Out}_{N_i}(H_i)$ and $\mathrm{Out}_{N}(G)$
are trivial,  and  (\ref{eq:wLift=Liftw}) reduces to Equation
(\ref{equationd}).
\item[b.] Let $G=\mathrm{Sp}_N$. Then $\times_{i=1,2}\mathrm{Out}_{N_i}(H_i)$
is a group of order two whose generator may be represented by the element
${w}= 
{w}_1\times \mathrm{I_{N_2}}
 \in \mathrm{O}_{N_1} \times \mathrm{Sp}_{N_2},
$ with $w_1$ as in (\ref{eq:representativeON}). Similarly, the automorphism
 ${}^{\vee}\mathrm{Int}(w)$ dual  
to $\mathrm{Int}(w)$ may be represented by the element 
${}^{\vee}{w}= 
{w}_1\times \mathrm{(-I_{N_2+1})}
 \in \mathrm{O}_{N_1}\times \mathrm{O}_{N_2+1}$.
Now, by Proposition \ref{prop:OutGofPi} (c)  with 
$\overline{s}=1$ 
$$
\etaABV_{\psi_{H_1}^{}}(H_1(\R))  \circ \mathrm{Int}(w_1) =   
\etaABV_{\mathrm{Int}({}^{\vee}w_1) \circ \psi_{H_1}}(H_1(\R))
$$
and so decomposition (\ref{eq:eta=tensor-eta}) permits us to write
$$
\etaABV_{\psi_{H}^{}}(H(\R)) \circ \mathrm{Int}(w) =  
\etaABV_{\mathrm{Int}({}^{\vee}w) \circ \psi_{H}}(H(\R)).
$$
Moreover, for $G=\mathrm{Sp}_{N}$ the element ${}^{\vee}{w}$ 
belongs to ${^{\vee}}{G}=\mathrm{SO}_{N+1}(\mathbb{C})$.  Consequently 
the Arthur parameters
$$\psi_G\, =\, \epsilon \circ \psi_{H}\quad\text{and}\quad
\mathrm{Int}\left({{}^{\vee}w}\right)\circ\psi_G\ \ =\ \epsilon \circ
\left( \mathrm{Int}\left({{}^{\vee}w} \right) \circ \psi_{H}\right)$$
are in the same ${^{\vee}}G$-orbit,
and  the  ${^\vee}G$-conjugate elements $\overline{s}$ and
$\mathrm{Int}({}^{\vee}w)(\bar{s})$ correspond to the same endoscopic
group $H(\mathbb{R})$.
The ordinary endoscopic transfer (\ref{equationd}) therefore implies
\begin{align*}
\mathrm{Lift}_{H(\R)}^{G(\R,\delta)}\left(\eta_{\mathrm{Int}({}^{\vee}w)
  \circ \psi_{H}}^{\mathrm{ABV}}(H(\R))\right) 
\ &=\ \eta_{\mathrm{Int}({}^{\vee}w)\circ\psi_{G}}^{\mathrm{ABV}}(\delta)
(\mathrm{Int}({}^{\vee}w)(\overline{s}))\\  
\ &=\  \eta_{\psi_{G}}^{\mathrm{ABV}}(\delta)(\bar{s}) \\
\ &=\  \mathrm{Lift}_{H(\R)}^{G(\R,\delta)} \left(
\eta_{\psi_{H}}^{\mathrm{ABV}}(H(\R))\right). 
\end{align*}
Since $\mathrm{Out}_N(G)$ is trivial in this case, we have proved
(\ref{eq:wLift=Liftw}).
\item[c.] Let $G=\mathrm{SO}_N$ where $N$ is even. We follow the same
  reasoning as in b.  Now
$\times_{i=1,2}\mathrm{Out}_{N_i}(H_i)$ is a group of order four whose
  generators may be represented by  elements
$$
w={w}_1\times \mathrm{Id}_{N_2}
 \in \mathrm{O}_{N_1} \times 
  \mathrm{SO}_{N_2}\ \text{ and }\
w'=\mathrm{Id}_{N_1}\times {w}_2
 \in \mathrm{SO}_{N_1}\times \mathrm{O}_{N_2}, 
 $$
 where $w_{1}$ and $w_{2}$ are as in (\ref{eq:representativeON}).
The same pair of elements can be used as generators for the dual
automorphisms.
Proposition \ref{prop:OutGofPi} (c) with 
$\overline{s}=1$ 
gives us
$$
\etaABV_{\psi_{H_i}^{}}(H_i(\R)) \circ \mathrm{Int}(w_i) \, =\,  
\etaABV_{\mathrm{Int}(w_i) \circ \psi_{H_i}}(H_i(\R)),
$$
and through decomposition  (\ref{eq:eta=tensor-eta}) 
we are able to deduce
\begin{align*}
\etaABV_{\psi_{H}}(H(\R)) \circ \mathrm{Int}(ww') \, &=\,  
\etaABV_{\mathrm{Int}({ww'}) \circ \psi_{H}}(H(\R)),
\quad \text{and}\\
\etaABV_{\mathrm{Int}({w'}) \circ
  \psi_{H}}(H(\R)) \circ \mathrm{Int}(ww')\, &=\,   
\etaABV_{\mathrm{Int}({w}) \circ \psi_{H}}(H(\R)).
\end{align*}
In addition, the product $ww'$ belongs to ${^{\vee}}{G}=\mathrm{SO}_N(\mathbb{C})$ and
$$
\mathrm{Int}\left({w}\right)\, \circ\,  \psi_G = 
\mathrm{Int}\left({w} w'\right) \circ \left(
\mathrm{Int}\left({w'}\right)\, \circ\,  \psi_G\right). 
$$
Hence, the Arthur parameters
$$
\psi_G\, =\, \epsilon \, \circ\,  \psi_{H} \quad \mbox{ and } \quad
\mathrm{Int}({w} w') \circ \psi_{G}, =\, \epsilon
\circ\left(\mathrm{Int}({w} w') \circ \psi_{H}\right) 
$$
belong to the same ${}^{\vee}G$-orbit, and
$$
 \mathrm{Int}\left({w}\right)\, \circ\,  \psi_G\, =\, \epsilon\, \circ\, 
\left( \mathrm{Int}( {w}) \circ \psi_{H} \right) \quad \mbox{ and } \quad
\mathrm{Int}\left({w'}\right)\, \circ\,  \psi_G =\, \epsilon\, \circ\,
\left( \mathrm{Int}({w'}) \circ \psi_{H} 
\right)
$$
belong to the same ${}^{\vee}G$-orbit. As in case b. the
${^\vee}G$-conjugate elements,
$\bar{s}$ and  $\mathrm{Int}(w w')(\bar{s})$,
correspond to the same endoscopic group $H(\mathbb{R})$.  Similarly,
the ${^\vee}G$-conjugate elements,  $\mathrm{Int}(w)(\bar{s})$ and
$\mathrm{Int}(w')(\bar{s})$,
correspond to the same endoscopic group $H(\mathbb{R})$. Therefore,
the endoscopic transfer map (\ref{equationd}) yields
\begin{align*}
\mathrm{Lift}_{H(\R)}^{G(\R,\delta)}\left(\eta_{\mathrm{Int}(w
  w') \circ \psi_{H}}^{\mathrm{ABV}}(H(\R))\right) 
\ &=\ \eta_{\mathrm{Int}(w w') \circ
  \psi_{G}}^{\mathrm{ABV}}(\delta)(\mathrm{Int}(w
w')(\overline{s}))\\ 
\ &=\ \eta_{\psi_{G}}^{\mathrm{ABV}}(\delta)(\overline{s})\\
&=\
\mathrm{Lift}_{H(\R)}^{G(\R,\delta)}\left(\eta_{\psi_{H}}^{\mathrm{ABV}}(H(\R))\right),\
\end{align*}
and likewise
$$\mathrm{Lift}_{H(\R)}^{G(\R,\delta)}\left(\eta_{\mathrm{Int}(w)
  \circ \psi_{H}}^{\mathrm{ABV}}(H(\R))\right) 
 = \mathrm{Lift}_{H(\R)}^{G(\R,\delta)}\left(\eta_{\mathrm{Int}(w')
   \circ \psi_{H}}^{\mathrm{ABV}}(H(\R))\right).$$ 
Finally, by Proposition \ref{prop:OutGofPi} (c), we have 
$$
w \cdot \eta_{\psi_{G}}^{\mathrm{ABV}}(\delta)(\overline{s})=
\eta_{\mathrm{Int}(w)
  \circ\psi_{G}}^{\mathrm{ABV}}(\delta)(\mathrm{Int}(w)(\overline{s})). 
$$
Since $\mathrm{Out}_N(G)$
is a group of order two  generated by $w$, the desired identity
(\ref{eq:wLift=Liftw}) follows. 
\end{enumerate}
\end{proof}
We can now lift Lemma \ref{lem:Out-of-eta} to our main theorem. 
\begin{thm}\label{thm:abv=ar}
Let $\psi_G^{}$ be an Arthur parameter for $G$. Then for any pure real
form $\delta$
of $G$
$$
\, \eta_{\psi_G^{}}^{\mathrm{Ar}}(\delta)(\overline{s})\ =\ 
\mathrm{Out}_N(G)\cdot\left(\eta_{\psi_G^{}}^{\mathrm{ABV}}(\delta)
(\overline{s})\right),\quad \overline{s}\in A_{\psi_G^{}}. 
$$
In particular, 
$$
\widetilde{\Pi}_{\psi_G^{}}^{\mathrm{Ar}}(\delta)\, =\, 
\mathrm{Out}_N(G)\cdot \Pi_{\psi_G^{}}^{\mathrm{ABV}}(\delta).
$$
In addition, for any $\pi\in  \Pi_{\psi_G^{}}^{\mathrm{ABV}}(\delta)$
\begin{align*}
\tau^{\mathrm{Ar}}_{\psi_G^{}}(\widetilde{\pi})\, =\,
\tau^{\mathrm{ABV}}_{\psi_G^{}}(\pi), \quad
\tau^{\mathrm{Ar}}_{\psi_{G}^{}}(\widetilde{\pi})(s_{\psi_{G}^{}})\,
=\, (-1)^{dS_\pi - dS_{\psi_{G}^{}}}, 
\end{align*}
with $\widetilde{\pi}$ the $\mathrm{Out}_N(G)$-orbit of $\pi$,
and $S_\pi$ the ${}^{\vee}G$-orbit in 
$\X$ corresponding to the
L-parameter of $\pi$. 
\end{thm} 
\begin{proof} 
Let $\bar{s}\in A_{\psi_G^{}}$. By Lemma \ref{lem:trans=lift}, Equation (\ref{eq:etaAr=etaABVodd}) 
 and Lemma \ref{lem:Out-of-eta}, we deduce
\begin{align*}
\eta_{\psi_G^{}}^{\mathrm{Ar}}(\delta)(\overline{s})&\ =\ 
\mathrm{Trans}_{H(\R)}^{G(\R,\delta)}(\eta_{\psi_{H}}^{\mathrm{Ar}}(H(\R)))\\
&\ =\ \mathrm{Lift}_{H(\R)}^{G(\R,\delta)}(\eta_{\psi_{H}}^{\mathrm{Ar}}(H(\R)))\\
&\ =\ \mathrm{Lift}_{H(\R)}^{G(\R,\delta)}\left(\left(\times_{i=1,2}\mathrm{Out}_{N_i}(H_i)\right)\cdot\eta_{\psi_{H}}^{\mathrm{ABV}}(H(\R))\right)\\ 
&\ =\ \mathrm{Out}_N(G)\cdot \eta_{\psi_G^{}}^{\mathrm{ABV}}(\delta)(\bar{s}).
\end{align*}
This proves the first assertion, and the identity of the Arthur packets
follows immediately.  Let $\widetilde{\pi} = \mathrm{Out}_N(G) \cdot
\pi$.  Then we may write
\begin{align*}
\sum_{\widetilde{\pi} \in \widetilde{\Pi}_{\psi_{G}^{}}^{\mathrm{Ar}}(\delta)} 
\mathrm{Tr} \left( \tau^{\mathrm{Ar}}_{\psi_{G}^{}} (\widetilde{\pi})
(s_{\psi_{G}}\bar{s}) \right)\,  \widetilde{\pi}
\ =\ \sum_{ \widetilde{\pi} \in \mathrm{Out}_N(G) \cdot \Pi_{\psi_{G}^{}}^{\mathrm{ABV}}(\delta)}
(-1)^{dS_\pi - dS_{\psi_{G}^{}}}
\ \mathrm{Tr}\left(\tau^{\mathrm{ABV}}_{\psi_{G}^{}}(\pi)(\bar{s})
\right) \, \widetilde{\pi}.
\end{align*} 
Since $\widetilde{\Pi}(G(\mathbb{R},\delta))$ is a basis of
$K_{\mathbb{C}} \widetilde{\Pi}(G(\mathbb{R},\delta))$, 
this equation implies 
$$\mathrm{Tr}
\left(\tau^{\mathrm{Ar}}_{\psi_{G}}(\widetilde{\pi})(s_{\psi_{G}} \bar{s})\right)
= (-1)^{dS_{\pi} - dS_{\psi_{G}}}
\ \mathrm{Tr} \left(\tau^{\mathrm{ABV}}_{\psi_{G}}(\pi)(\bar{s}) \right).$$
This may be regarded as an equality between virtual characters on
$A_{\psi_G}$, and so the linear independence of these characters implies
the last assertions of the theorem.
\end{proof}

\end{document}